\documentclass[11pt]{article}
\usepackage{amsmath, amssymb, amscd, amsthm, amsfonts, tikz}
\usepackage{graphicx}
\usepackage{hyperref}
\usepackage{float}
\usepackage{comment}
\usepackage{caption}
\usepackage{subcaption}
\usepackage{ytableau}
\usetikzlibrary{patterns}

\title{On Semisymmetric Height and a Multidimensional Generalization of Weighted Catalan Numbers}
\author{Ryota Inagaki \and Dimana Pramatarova}

\newtheorem{theorem}{Theorem}[section]
\newtheorem{corollary}[theorem]{Corollary}

\newtheorem{proposition}[theorem]{Proposition}

\newtheorem{remark}[theorem]{Remark}

\theoremstyle{remark}

\theoremstyle{definition}
\newtheorem{definition}[theorem]{Definition}
\newtheorem{example}[theorem]{Example}

\newtheorem{openproblem}[theorem]{Open Problem}

\allowdisplaybreaks

\begin{document}
\maketitle

\begin{abstract}
Weighted Catalan numbers are a class of weighted sums over Dyck paths. Well-studied for their arithmetic properties and applications to enumerative combinatorics, these numbers were recently generalized to the setting of $k$-dimensional Catalan numbers for $k \geq 2$. In this paper, we introduce the $k$-dimensional semisymmetric weighted Catalan numbers ($k$-dimensional SSWCNs), an alternative $k$-dimensional generalization, along with their variant, the $k$-dimensional $u$-bounded semisymmetric weighted Catalan numbers ($k$-dimensional $u$-bounded SSWCNs). We define these two classes of numbers using the notion of semisymmetric height, a new statistic on points in $\mathbb{Z}^k_{\geq 0}$ motivated by geometric symmetries of $k$-dimensional analogs of Dyck paths and of the fundamental Weyl chamber of type $A_{k-1}$. For our main results, we prove the eventual periodicity of $k$-dimensional SSWCNs and their $u$-bounded variants modulo a suitable integer $m$, and we derive formulas for several classes of $k$-dimensional $u$-bounded SSWCNs. Additionally, using semisymmetric height, we derive novel analogs in the $k$-dimensional setting of the integer sequence counting Dyck paths by height and of the Narayana numbers. We conclude the paper with a future direction for generalizing weighted Catalan numbers to the $k$-dimensional setting.
\end{abstract}
\renewcommand{\thefootnote}{\fnsymbol{footnote}} 



\maketitle

\section{Introduction}
\label{section:intro}
\subsection{Catalan Numbers}
The sequence of Catalan numbers $\displaystyle C_n = \frac{1}{n+1} \binom{2n}{n}$ is well studied in enumerative and algebraic combinatorics (A$000108$ in the OEIS \cite{oeis}). The $n$th Catalan number counts the number of Dyck paths of $2n$ steps. A Dyck path is a sequence of steps $(\vec{s}_1, \vec{s}_2, \dots, \vec{s}_{2n})$ in $\mathbb{Z}^2$ that starts at $(0,0)$ and ends at $(2n,0)$. It uses $n$ steps of the form $(1,1)$, which are called \textit{up-steps}, and $n$ of the form $(1, -1)$, which are called \textit{down-steps}. Each intermediate point $\vec{0}+\sum_{j=1}^i\vec{s}_j$ must remain on or above the $x$-axis. The Catalan numbers have many refinements, such as the Narayana numbers (A001263 in the OEIS \cite{oeis}), which count Dyck paths by the number of peaks, where a {\em peak} is an up-step immediately followed by a down-step. The Narayana number $N(n, \alpha)$ counts the Dyck paths of $2n$ steps with exactly $\alpha$ peaks.

\subsection{Weighted Catalan Numbers}\label{sec:WCN}
We next discuss the weighted Catalan numbers, introduced by Goulden and Jackson \cite{Goulden}, which generalize Catalan numbers via weighted sums over Dyck paths. This geometric perspective motivates our later $k$-dimensional generalizations, defined as weighted sums over $k$-dimensional analogs of Dyck paths.

Weighted Catalan numbers are defined as follows. For a sequence of integers $\vec{b} = (b_0,b_1,b_2,\ldots)$ and a Dyck path $P$ of $2n$ steps, the \textbf{weight} $wt_{\vec{b}}(P)$ is defined as the product $b_{h_1} b_{h_2}\cdots b_{h_n}$, where each $h_i$ is the $y$-coordinate of the starting point for the $i$th up-step in $P$. Accordingly, the \textbf{$n$th weighted Catalan number} is $$C_n^{\vec{b}} = \sum_P wt_{\vec{b}}(P),$$ with the sum over all Dyck paths of length $2n$. In particular, taking $\vec{b}=(1, 1, \dots)$, we recover the standard Catalan number $C_n$.
 
The periodicity of the weighted Catalan numbers modulo an integer has been extensively studied. In \cite{postnikov2000counting}, Postnikov conjectured that the sequence $C^{(1^2, 3^2, 5^2, \dots)}_n \pmod{3^r}$ is periodic with a period of $2\cdot3^{r-3}$. In 2013, Shader \cite{shader2013weighted} found an upper bound on the period of $C^{(1^2, 3^2, 5^2, \dots)}_n \pmod{3^r}$. In 2021, Gao and Gu \cite{gaogu} proved Postnikov's conjecture and found a necessary and sufficient condition for the eventual periodicity of $C_n^{\vec{b}}$ modulo an integer.

\subsection{Relation to the \texorpdfstring{$k$}{k}-dimensional Weighted Catalan Numbers}
In our previous work \cite{MR5034574}, we extended the study of weighted Catalan numbers to the $k$-dimensional case for $k \geq 2$, expanding upon the foundational results described above.

To achieve this extension, we employed $k$-dimensional balanced ballot paths, which are $k$-dimensional generalizations of Dyck paths. These paths are lattice paths on $\mathbb{Z}_{\geq 0}^k$  whose steps are standard unit vectors (i.e., the unit vectors $\vec{e}_1$, $\vec{e}_2$, etc). (We formally define balanced ballot paths in Section~\ref{sec:bbp}.) For $k=2$, these paths correspond to Dyck paths as each $\vec{e}_1$ step becomes $(1,1)$ and each $\vec{e}_2$ step becomes $(1,-1)$. Figure~\ref{fig:PathExample} shows an example of a 3-dimensional balanced ballot path.

We also defined a notion of height in \cite{MR5034574}. For $k\geq 2$, we defined the height of points in  $\mathbb{Z}_{\geq 0}^k$ as \[h_k(\vec{x}) = (k-1)x_1-x_2-x_3 -\dots -x_k.\] We then defined, for any $k$-dimensional balanced ballot path $P$, the height $h_k(P)$ of $P$ as the maximum value of $h_k(\vec{x})$ over all intermediate points of $P$. We chose this height function because it specializes at $k=2$ to the height of intermediate points on Dyck paths.

Using the above definition of height $h_k$, we defined \textbf{$k$-dimensional weighted Catalan numbers}, the {\em first} generalization of the weighted Catalan numbers to the $k$-dimensional setting, as a weighted sum over
$k$-dimensional balanced ballot paths. In \cite{MR5034574}, we computed examples and studied the arithmetic properties of these numbers. As additional applications of the height function $h_k$, we introduced analogs for the enumeration of Dyck paths by height and the Narayana numbers. 

\begin{figure}[H]
    \centering
    \includegraphics[width=0.40\linewidth]{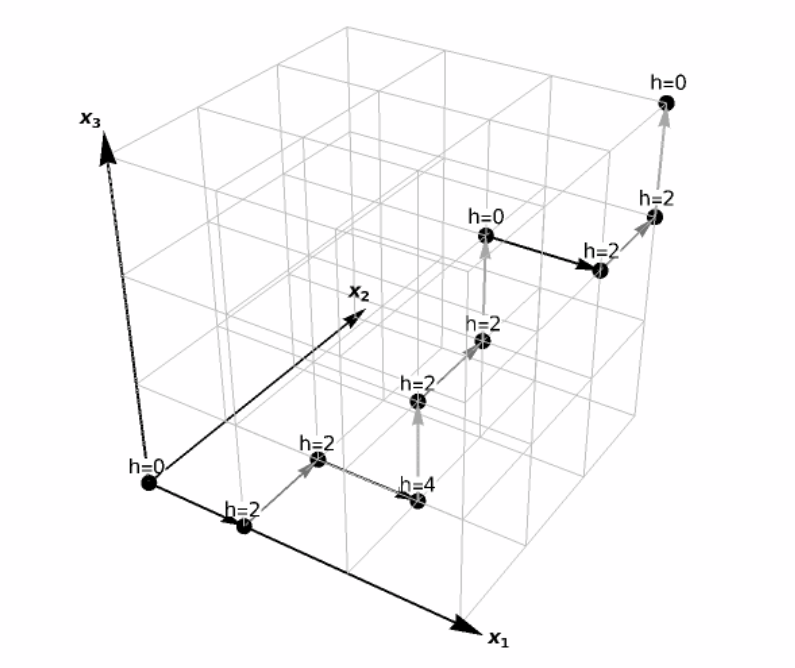}
    \caption{A $3$-dimensional balanced ballot path of $9$ steps. Each intermediate point of the path is labeled with its semisymmetric height $g_3(\vec{x}) = 2x_1-2x_3$.}
    \label{fig:PathExample}
\end{figure}

\subsection{Our Motivation and Contributions}

Although the height function $h_k$ from \cite{MR5034574} successfully generalized the notion of height of Dyck paths and enabled us to define $k$-dimensional weighted Catalan numbers, it does not exhibit symmetries of its specialization  $h_2(\vec{x}) = x_1- x_2$, the height of intermediate points in Dyck paths. More specifically, for any nonnegative integer $n$, the function $h_k$ is not invariant under the involution of $k$-dimensional points $\phi_{k, n}(\vec{x}) = (n-x_k, n-x_{k-1}, \dots, n-x_{1})$. (For example, observe that $h_3(\phi_{3, 1}(1, 0, 0)) = h_3((1, 1, 0)) = 1 \neq 2= h_3((1, 0, 0))$, whereas we have $h_2 \circ \phi_{2, n} = h_2$ for all nonnegative integers $n$.) Furthermore, it does not account for the symmetries of the fundamental Weyl chamber of type $A_{k-1}$, the region $\{\vec{x} \in \mathbb{Z}_{\geq 0}^k: x_1 \geq x_2 \geq \dots \geq x_k\}$ containing the intermediate points of balanced ballot paths. From the perspective of Weyl chambers, the map $\phi_{k, n}$ is also a natural reflection of coordinates. (We will more formally discuss the Weyl chamber in Section~\ref{sec:ssheight}.)

In light of this lack of symmetry exhibited by $h_k$, we introduce a new notion of height of points in $\mathbb{Z}^k$ and balanced ballot paths, the \textbf{semisymmetric height}. Following a suggestion made by our colleague Kenta Suzuki \cite{KSuzuki}, we define the semisymmetric height of the vector $\vec{x} \in \mathbb{Z}^k_{\geq 0}$ as follows: $$g_k(\vec{x}) = (k-1)x_1 + (k-3)x_2 + \cdots + (3-k)x_{k-1} + (1-k)x_k = \sum_{i=1}^k (k+1-2i)x_i.$$ 
 The semisymmetric height of a $k$-dimensional balanced ballot path $P$ is the maximum semisymmetric height attained by an intermediate point in $P$. This function, unlike $h_k$, is both invariant under $\phi_{k, n}$ for all $n$ and respects symmetries of the fundamental Weyl chamber of type $A_{k-1}$. We name this function the \textit{semisymmetric} height since the function $g_k$ exhibits the aforementioned symmetries, but is not a symmetric function in the sense of Stanley \cite[Chapter 7]{MR4621625}.

Using the semisymmetric height, we introduce an alternative generalization of weighted Catalan numbers to the setting of $k$-dimensional balanced ballot paths. More specifically, we introduce two new objects: the \textbf{$k$-dimensional semisymmetric weighted Catalan numbers ($k$-dimensional SSWCNs)}—the main $k$-dimensional generalization of weighted Catalan numbers for this paper—and their variant, the \textbf{$k$-dimensional $u$-bounded semisymmetric weighted Catalan numbers ($k$-dimensional $u$-bounded SSWCNs)}. Both are weighted sums over $k$-dimensional balanced ballot paths, with the latter being a sum over $k$-dimensional balanced ballot paths whose semisymmetric heights are at most $u$. (See Table~\ref{tab:notations} for a list of notations and links to definitions for $k$-dimensional SSWCNs and their variants.) The introduction of these two objects is the {\em central conceptual contribution}, around which we develop subsequent results.

For this paper's main results, we prove the eventual periodicity of the $k$-dimensional SSWCNs and their $u$-bounded variants under suitable conditions on $m$ (Theorems~\ref{thm:periodboundedkdim},~\ref{thm:period-weighted-kdim}, and ~\ref{thm:Productperiodicity}), and we derive recursive and closed formulas for several classes of $k$-dimensional $u$-bounded SSWCNs involving small $k$ and $u$. As special cases of formulas for $k$-dimensional $u$-bounded SSWCNs, we obtain numerous notable corollaries, such as the result that the number of $3$-dimensional balanced ballot paths of length $3n$ and semisymmetric height at most $4$ is the $n$th term of the OEIS sequence A015448 \cite{oeis} (Proposition~\ref{cor:OEISEquals}). While the proofs of our main results follow a similar approach to those in Sections 3 and 4 of \cite{MR5034574}, {\em the results themselves are new, as they describe the $k$-dimensional SSWCNs and their $u$-bounded variants.}

For our secondary results and additional applications of semisymmetric height, we use semisymmetric height to develop novel analogs of integer sequences. We introduce the \textbf{$k$-dimensional semisymmetric height triangle}, a novel $k$-dimensional analog of the sequence counting Dyck paths by their exact height (A080936 in the OEIS \cite{oeis}). By defining a new notion of peaks of balanced ballot paths via semisymmetric height, we also define the $k$-dimensional semisymmetric Narayana numbers, a new analog of the Narayana numbers. For the last two analogs, we compute numerical data showing that they differ from the corresponding sequences in Sulanke's paper \cite{multidimnarayana} and in our prior paper \cite{MR5034574}. We show that the $k$-dimensional semisymmetric height and Narayana triangles share a subsequence (Theorem~\ref{thm:RelatingTheTriangles}).

\subsection{Roadmap}

 In Section \ref{section:preliminaries}, we formally introduce definitions and basic properties used throughout the paper. We begin by defining balanced ballot paths, semisymmetric height, and semisymmetric weights of such paths. Then, we use those notions to introduce the main subjects of this paper: the $k$-dimensional SSWCNs and the $k$-dimensional $u$-bounded SSWCNs (Definitions~\ref{defn:kdimweighted} and~\ref{defn:k-dims-boundCat}). We then prove that the $k$-dimensional SSWCNs are distinct from the $k$-dimensional weighted Catalan numbers in \cite{MR5034574}. Afterward, we find bounds on semisymmetric heights of balanced ballot paths and introduce auxiliary variants of SSWCNs for use in Sections~\ref{sec:periodicity} and~\ref{sec:kuboundedweightedCatalan}.

In Section~\ref{sec:periodicity}, we prove periodicity of $k$-dimensional SSWCNs and $k$-dimensional $u$-bounded SSWCNs under suitable conditions.

In Section~\ref{sec:kuboundedweightedCatalan}, we prove formulas for several examples of $k$-dimensional $u$-bounded SSWCNs.

In Section~\ref{sec:heighttriangles}, we introduce the $k$-dimensional semisymmetric height triangle,  the integer sequence counting $k$-dimensional balanced ballot paths by their exact semisymmetric height. 

In Section~\ref{sec:Narayana}, we introduce and study the $k$-dimensional semisymmetric Narayana triangle. 

In Section~\ref{sec:FurtherDir}, we propose a future direction for discovering novel $k$-dimensional generalizations of weighted Catalan numbers.

\section{Definitions and Preliminaries}
\label{section:preliminaries}

In this section, we introduce the key definitions and preliminary results used throughout the paper.

In Section~\ref{sec:bbp}, we recall the definition of balanced ballot paths from \cite{MR5034574}. These paths serve as the underlying combinatorial objects for our constructions. We also introduce a new taxonomy for steps in such paths. In Section~\ref{sec:ssheight}, we introduce and motivate the semisymmetric height and define the associated notion of semisymmetric weight. In Section~\ref{sec:MainObject}, we introduce the main objects: the $k$-dimensional semisymmetric weighted Catalan numbers (SSWCNs) and their $u$-bounded variants ($k$-dimensional $u$-bounded SSWCNs). The definitions developed in the preceding subsections provide the necessary framework for these main objects. In Section~\ref{sec:Relation}, we show that these objects are distinct from the $k$-dimensional weighted Catalan numbers of \cite{MR5034574}. The remaining subsections introduce auxiliary definitions and results used in later proofs. 

Throughout, we introduce various notations; for clarity, Table~\ref{tab:notations} summarizes the notation for the $k$-dimensional SSWCNs and their key variants.
\subsection{Balanced Ballot Paths}\label{sec:bbp}
We first state definitions related to balanced ballot paths, a generalization of Dyck paths to $k$-dimensional space. These are our core combinatorial objects. Many definitions in this section are from our previous paper \cite{MR5034574} and are included here for completeness. Balanced ballot paths provide the basis for defining the semisymmetric height, as well as the two main objects of the paper: the $k$-dimensional SSWCNs and $k$-dimensional $u$-bounded SSWCNs, discussed next in Section~\ref{sec:ssheight}. We then define these main objects as weighted sums over balanced ballot paths, following the geometric perspective on weighted Catalan numbers, where the weights are the sums of the weights of the corresponding Dyck paths.

We begin with the basic terminology for paths, most of which was defined at the beginning of Section 2.1 of \cite{MR5034574}. In this paper, we assume all vectors are in the $k$-dimensional Euclidean space $\mathbb{R}^k$ for $k \geq 2$. The vector $\vec{e}_i$ is the $i$th standard basis vector in $\mathbb{R}^k$. A point in the $k$-dimensional lattice $\mathbb{Z}^k$ is a $k$-tuple $\vec{x} = (x_1, x_2, \ldots, x_k)$. A path of $\ell$ steps in the $k$-dimensional lattice is a finite sequence $P= (\vec{s}_1, \vec{s}_2, \dots, \vec{s}_{\ell})$ of vectors in $\mathbb{R}^k$ called steps. Given a path $P$ and $i \in \{0, 1, \dots, \ell\}$, the $i$th intermediate point of $P$ is $\vec{v}_i = \vec{0}+\sum_{j=1}^i \vec{s}_j$. The starting point of the $i$th step is the $(i-1)$th intermediate point, and the resulting point is the $i$th intermediate point.

We now formally state the definition of balanced ballot paths, the generalization of Dyck paths to the $k$-dimensional Euclidean space. We imported this definition from Section 2 of \cite{MR5034574}.

\begin{definition}
\label{defn:k-dim-path}
A $k-$\textbf{dimensional balanced ballot path} is a path $P = (\vec{s}_1, \vec{s}_2, \dots, \vec{s}_{kn})$ of $kn$ steps satisfying the following properties:
\begin{itemize}
\item (\textbf{balanced property}) For each $i \in \{1, 2, \dots, k\}$, exactly $n$ steps are equal to $\vec{e}i$.
\item (\textbf{ballot property}) For each intermediate point $\vec{x} = (x_1, x_2, \dots, x_k) = \sum_{j=1}^{i}\vec{s}_j$, we have $x_1 \geq x_2 \geq \dots \geq x_k$.
\end{itemize}
\end{definition}
\begin{example}\label{ex:3Dpathexample}
The path $P = (\vec{e}_1,\vec{e}_2,\vec{e}_1, \vec{e}_3,\vec{e}_2,\vec{e}_3, \vec{e}_1,\vec{e}_2,\vec{e}_3)$ is a $3$-dimensional balanced ballot path of $9$ steps. Figure~\ref{fig:PathExample} illustrates this path. In this section, we frequently refer to $3$-dimensional balanced ballot paths.
\end{example}
In particular, the $2$-dimensional balanced ballot paths are equivalent to Dyck paths \cite{StanleyCatalan}.  For any $2$-dimensional balanced ballot path $P$, we obtain the equivalent Dyck path by replacing each $\vec{e}_1$ step with $(1, 1)$ and each $\vec{e}_2$ step with $(1, -1)$.

\begin{remark}
We name the ballot property of $k$-dimensional balanced ballot paths in Definition~\ref{defn:k-dim-path} after ballot sequences (Chapter 7 of \cite{StanleyCatalan}), which are similarly defined sequences of integers.
\end{remark}

Using balanced ballot paths, we introduce the $k$-dimensional Catalan numbers, and we find weighted analogs.
\begin{definition}[A$060854$ in OEIS \cite{oeis}]\label{def:multidim}
For $n$ and $k$, the $n$-th $k$-dimensional Catalan number $C_{k, n}$ is the number of $k$-dimensional balanced ballot paths of length $kn$.
\end{definition}

The $n$th $k$-dimensional Catalan number equals $$C_{k,n} = \frac{0!1!\cdots(n-1)! \cdot (kn)!}{k!(k+1)!\cdots(k+n-1)!}.$$ In particular, we have $C_{2, n} = C_n$.
\begin{example}\label{ex:EnumerationRunning}
We find that $C_{3, 2} = 5$ because there are exactly five 3-dimensional balanced ballot paths of length 6. They are $(\vec{e}_1, \vec{e}_1, \vec{e}_2, \vec{e}_2, \vec{e}_3, \vec{e}_3)$, $(\vec{e}_1, \vec{e}_2, \vec{e}_1, \vec{e}_2, \vec{e}_3, \vec{e}_3)$, $(\vec{e}_1, \vec{e}_1, \vec{e}_2, \vec{e}_3, \vec{e}_2, \vec{e}_3)$, $(\vec{e}_1, \vec{e}_2, \vec{e}_3, \vec{e}_1, \vec{e}_2, \vec{e}_3)$, and $(\vec{e}_1, \vec{e}_2, \vec{e}_1, \vec{e}_3, \vec{e}_2, \vec{e}_3)$. Figure~\ref{fig:five-balanced-ballot-paths} illustrates these balanced ballot paths. These paths will appear later in Example~\ref{ex:computation}.
\begin{figure}[H]
\centering

\begin{minipage}{0.32\textwidth}
    \centering
    \includegraphics[width=0.85\linewidth]{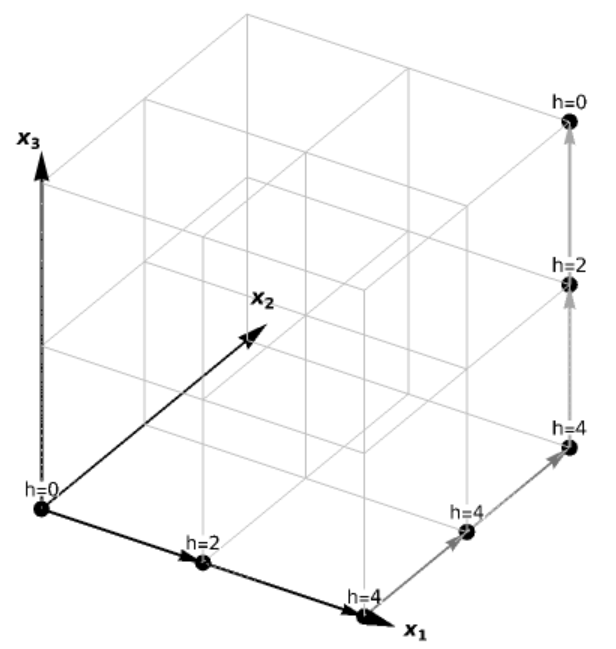}
    \caption*{$(\vec{e}_1, \vec{e}_1, \vec{e}_2, \vec{e}_2, \vec{e}_3, \vec{e}_3)$.}
\end{minipage}
\hfill
\begin{minipage}{0.32\textwidth}
    \centering
    \includegraphics[width=0.85\linewidth]{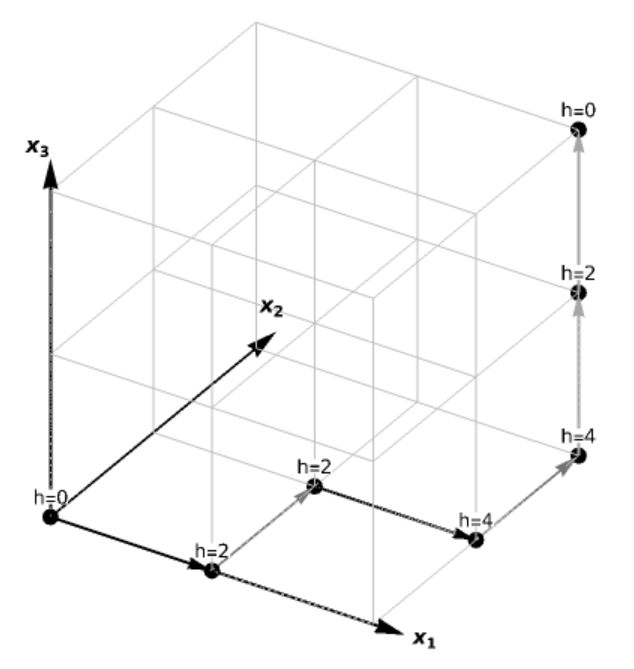}
    \caption*{$(\vec{e}_1, \vec{e}_2, \vec{e}_1, \vec{e}_2, \vec{e}_3, \vec{e}_3)$.}
\end{minipage}
\hfill
\begin{minipage}{0.32\textwidth}
    \centering
    \includegraphics[width=0.85\linewidth]{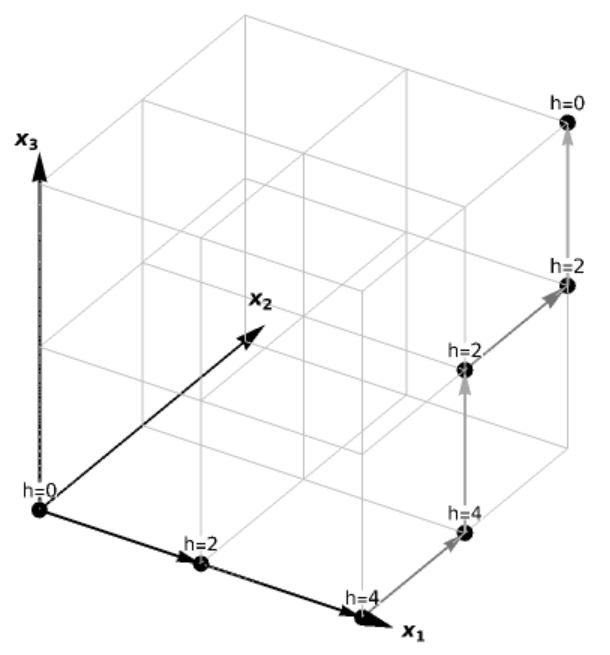}
    \caption*{$(\vec{e}_1, \vec{e}_1, \vec{e}_2, \vec{e}_3, \vec{e}_2, \vec{e}_3)$.}
\end{minipage}

\vspace{0.8em}

\begin{minipage}{0.32\textwidth}
    \centering
    \includegraphics[width=0.85\linewidth]{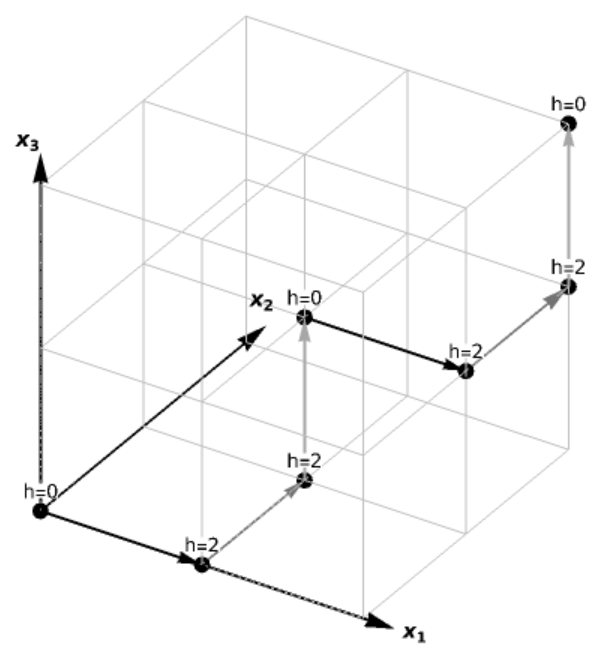}
    \caption*{$(\vec{e}_1, \vec{e}_2, \vec{e}_3, \vec{e}_1, \vec{e}_2, \vec{e}_3)$.}
\end{minipage}
\hfill
\begin{minipage}{0.32\textwidth}
    \centering
    \includegraphics[width=0.85\linewidth]{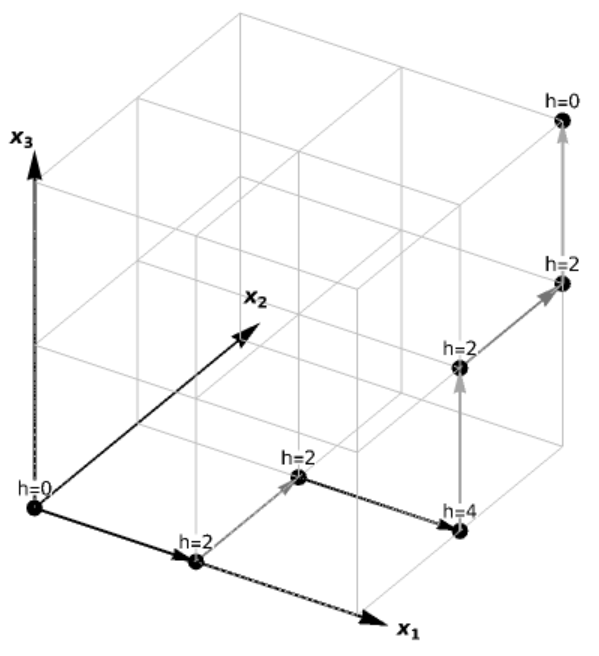}
    \caption*{$(\vec{e}_1, \vec{e}_2, \vec{e}_1, \vec{e}_3, \vec{e}_2, \vec{e}_3)$.}
\end{minipage}

\caption{The five $3$-dimensional balanced ballot paths of length $6$, with intermediate points labeled by their semisymmetric height.}
\label{fig:five-balanced-ballot-paths}

\end{figure}
\end{example}

\begin{remark}
The $n$th $k$-dimensional Catalan number is also equal to the number of standard Young tableaux of shape $k \times n$ (See Chapter 7 of Stanley \cite{MR4621625}).
\end{remark}

We also recall the definition of sub-ballot paths from \cite{MR5034574}. The sub-ballot paths will be crucial for our proofs in Sections~\ref{sec:periodicity} and~\ref{sec:kuboundedweightedCatalan}.
\begin{definition}
\label{defn:k-dim-pathsubballot}
Let $\vec{a}, \vec{a}' \in \mathbb{Z}_{\geq 0}^{k}$. A \textbf{$k$-dimensional sub-ballot path from $\vec{a}$ to $\vec{a}'$} is a path $P =(\vec{s}_1, \vec{s}_2, \dots, \vec{s}_{\ell})$ satisfying the following properties:
\begin{itemize}
\item Each step $\vec{s}_i$ is in ${\vec{e}_1, \vec{e}_2, \dots, \vec{e}_k}$.
\item For each $i \in \{0,1, \dots, \ell\}$, the point $\displaystyle (x_1, x_2, \dots, x_k) = \vec{a}+\sum_{j=1}^i\vec{s}_j$ satisfies the property $x_1 \geq x_2 \geq \dots \geq x_k\geq 0$.
\item We have $\displaystyle \vec{a} + \sum_{j=1}^{\ell}\vec{s}_j = \vec{a}'$.
\end{itemize}
For a $k$-dimensional sub-ballot path $P = (\vec{s}_1, \vec{s}_2, \dots, \vec{s}_{\ell})$ from $\vec{a}$ to $\vec{a}'$ and $i \in \{1, 2, \dots, \ell\}$, the \textbf{starting point of the $i$th step} is $\vec{a} + \sum_{j=1}^{i-1}\vec{s}_j$. The \textbf{resulting point of the $i$th step} is $\vec{a} + \sum_{j=1}^{i}\vec{s}_j$.
\end{definition}
In particular, $k$-dimensional balanced ballot paths with $kn$ steps are equivalent to $k$-dimensional sub-ballot paths from $(0, 0, \dots, 0)$ to $(n, n, \dots, n)$.

For both balanced ballot and sub-ballot paths, we define semisymmetric up-steps, down-steps, and neutral steps: a new taxonomy of steps that generalizes the up-steps and down-steps of Dyck paths. Our classification of steps will allow us to define semisymmetric weight and, consequently, the main objects of our paper in Section~\ref{sec:MainObject}. We also apply the taxonomy of steps to define an analog of peaks for $k$-dimensional balanced ballot paths in Section~\ref{sec:Narayana}.

\begin{definition}
\label{defn:up-step-down-step}
We call a \textbf{semisymmetric up-step} any step in the direction of $\vec{e}_i$, where $i \in \{1,2,\ldots, \left\lfloor \frac{k}{2}\right\rfloor \}.$ Similarly, a \textbf{semisymmetric down-step} is any step equal to $\vec{e}_j,$ where $j \in \{k , k-1, \dots, k - \left\lfloor \frac{k}{2} \right\rfloor+1\}$. Finally, for odd $k$, we define the \textbf{semisymmetric neutral step}, which is the step in the direction of $\vec{e}{\left \lfloor\frac{k}{2}\right\rfloor+1}$. Semisymmetric neutral steps do not exist for even $k$.
\end{definition}
We refer to these steps as semisymmetric up-steps and down-steps because they increase and decrease the semisymmetric height $g_k$, respectively, in analogy with the role of up-steps and down-steps in Dyck paths.

In particular, a $k$-dimensional balanced ballot path of length $kn$ contains $\left\lfloor \frac{k}{2}\right\rfloor n$ semisymmetric up-steps and $\left\lceil \frac{k}{2}\right\rceil n$ steps that are either neutral or down-steps.
\subsection{The Semisymmetric Height and Weight}\label{sec:ssheight}
In this subsection, we introduce and motivate the semisymmetric height and then define the semisymmetric weight, which will be used in Section~\ref{sec:MainObject} to define the $k$-dimensional SSWCNs and their $u$-bounded variants.

Recall the height function $$h_k(\vec{x}) = (k-1)x_1 -x_2-\dots-x_k$$ from \cite{MR5034574}. The height function $h_k$ is a key tool for $k$-dimensional balanced ballot paths. However, for $k \geq 3$, this function does not respect a natural symmetry of these paths, namely the involution $\phi_{k,n}$ defined below. Since this symmetry arises naturally from both the geometry of balanced ballot paths of the fundamental Weyl chamber of type $A_{k-1}$, it is desirable to work on a height function that is invariant under $\phi_{k,n}$. This motivates the introduction of the semisymmetric height.

To that end, consider the map $\phi_{k, n}(x_1,\ldots,x_k)= (n-x_k,\ldots,n-x_1)$. The map is significant for two reasons. First, it is an involution on $k$-dimensional balanced ballot paths of length $kn$. More specifically, for any sequence of intermediate points $\vec{v}_0, \vec{v}_1, \vec{v}_2, \dots, \vec{v}_{kn}$ of a $k$-dimensional balanced ballot path $P$ of length $kn$, the sequence $\phi{k, n}(\vec{v}_{kn}), \phi_{k, n}(\vec{v}_{kn-1}), \dots, \phi_{k, n}(\vec{v}_0)$ is a sequence of intermediate points corresponding to some other balanced ballot path. Second, the map $\phi{k, n}$ is a natural reflection of points in the fundamental Weyl Chamber of type $A_{k-1}$. In the theory of reflection and Coxeter groups (see \cite{MR1066460} for a formal exposition), the \textbf{fundamental Weyl chamber of type $A_{k-1}$} is the region $\{\vec{x} \in \mathbb{Z}_{\geq0 }^k: x_1 \geq x_2 \geq \dots \geq x_k\}$.

The height function for $k=2$, $h_2$, is invariant under $\phi_{2, n}$. Indeed, we have $(h_2 \circ \phi_{2, n})(x, y) = h_2(n-y, n-x)= x-y=h_2(x, y)$. On the other hand, for $k \geq 3$, the invariance does not hold. This can be seen by computing $(h_k \circ \phi_{k, n})(x_1, x_2, \dots, x_k)  = x_1+x_2+\dots+x_{k-1}-(k-1)x_{k} \neq h_k(x_1, x_2, \dots, x_k)$. In light of this lack of symmetry exhibited by $h_k$, we define a new notion of height that is invariant under $\phi_{k, n}$.
\begin{definition}
\label{def:height-new}
For points in $\mathbb{Z}{\geq 0}^k$, the \textbf{semisymmetric height} function $g_k: \mathbb{Z}_{\geq 0}^k \to \mathbb{R}$ is $$g_k(\vec{x}) = (k-1)x_1 + (k-3)x_2 + \cdots + (3-k)x_{k-1} + (1-k)x_k = \sum_{i=1}^k (k+1-2i)x_i.$$ For any intermediate point $\vec{x}$ in a $k$-dimensional balanced ballot path $P$, the \textbf{semisymmetric height of the intermediate point} $\vec{x}$ is $g_k(\vec{x})$.  For any $k$-dimensional balanced ballot path $P = (\vec{s}_1, \vec{s}_2, \dots, \vec{s}_{kn})$, the semisymmetric height of $P$ is denoted by $$g_k(P) = \max_{i \in \{0,1, \dots, kn\}}{g_k(\vec{0}+\sum_{j=1}^i\vec{s}i)}.$$ Similarly, for $\vec{a}, \vec{a}^{\prime} \in \mathbb{Z}_{\geq 0}^k $ and any $k$-dimensional sub-ballot path $P = (\vec{s}_1, \vec{s}_2, \dots, \vec{s}_{\ell})$ from $\vec{a}$ to $\vec{a}^{\prime}$, the \textbf{semisymmetric height of sub-ballot path $P$} is $$g_k(P) = \max_{i \in \{0,1, \dots, \ell\}}{g_k(\vec{a}+\sum_{j=1}^i\vec{s}_i)}.$$ As a convention, the semisymmetric height of the empty balanced ballot path $P = \emptyset$ is $0$.
\end{definition}

The definition of the semisymmetric height function was suggested by Kenta Suzuki~\cite{KSuzuki}.

\begin{example}\label{ex:Example3Dbalancedballot}
Consider the $3$-dimensional balanced ballot path $P = (\vec{e}_1, \vec{e}_1, \vec{e}_2, \vec{e}_2, \vec{e}_3, \vec{e}_3)$. We illustrate this path in Figure~\ref{fig:3ddyck}, where each intermediate point $\vec{x}$ is labeled with its semisymmetric height $g_k(\vec{x})$.
\begin{figure}[H]
\centering
\includegraphics[width=0.35\linewidth]{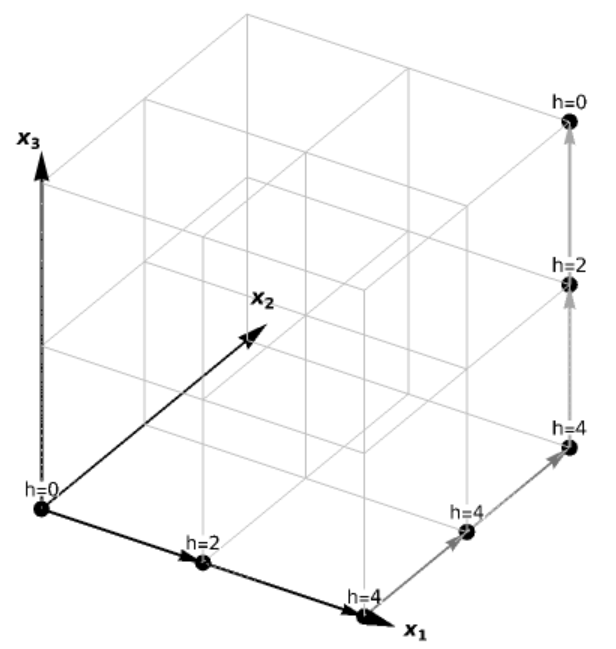}
\caption{A 3-dimensional balanced ballot path. We label each intermediate point by its semisymmetric height $g_3(\vec{x}) = 2x_1-2x_3$.}
\label{fig:3ddyck}
\end{figure}
\end{example}

The semisymmetric height is a natural generalization of the Dyck path height. Indeed, in the case $k=2$, we have
\[g_2(\vec{x}) = x_1 - x_2,\]
which coincides with the usual height function for intermediate points of Dyck paths.

We show that the semisymmetric height $g_k$ is designed to respect the symmetries of the fundamental Weyl chamber of type $A_{k-1}$. In the $k$-dimensional Euclidean space, the region where the balanced ballot paths reside, $\{\vec{x} \in \mathbb{Z}_{\geq0 }^k: x_1 \geq x_2 \geq \dots \geq x_k\}$, can be viewed as the fundamental Weyl chamber of type $A_{k-1}$. The involution $\phi_{k, n}$ is a natural reflection of points in $C$ whose coordinates are at most $n$.

We can formalize this by using the positive roots of type $A_{k-1}$, namely $\{\vec{e}_j - \vec{e}_{j'}: j < j' \}$. This naturally yields the coefficients of $g_k(\vec{x})$, via the half-sums of positive roots $\frac{1}{2}(\sum_{j=1}^{i-1}(\vec{e}_j - \vec{e}i) +\sum_{j=i+1}^k(\vec{e}_i - \vec{e}_j))$.

Finally, we observe the symmetries in the formula for $g_k$. For $i \in \{1, 2, \dots, k\}$, the coefficient of $x_i$ in $g_k(\vec{x})$ is the negation of that of $x_{k+1-i}$. This is a symmetry not present in the expression of $h_k(\vec{x})$. Because $g_k$ exhibits this and the previous noted symmetries but is not a symmetric function in the classical sense (c.f., Chapter 7 of Stanley \cite{MR4621625}), we call $g_k$ the \textit{semisymmetric} height.

Just as counting Dyck paths below a given height refines the Catalan numbers (sequence A080934 in the OEIS \cite{oeis}), we introduce an analogous refinement of balanced ballot paths using semisymmetric height. In Section~\ref{sec:kuboundedweightedCatalan}, we derive closed formulas for several cases of these numbers, and in Section~\ref{sec:heighttriangles}, they underlie the $k$-dimensional semisymmetric height triangles.
\begin{definition}
\label{def:bounded-and-height}
We define the \textbf{$n$th $k$-dimensional $u-$bounded semisymmetric Catalan number} $\widehat{C}_{k,u,n}$ as the number of balanced ballot paths of length $kn$ with semisymmetric height at most $u$.
\end{definition}
\begin{example}
We compute $\widehat{C}_{3, 2, 2}$. Out of the $C_{3 , 2}=5$ $3$-dimensional balanced ballot paths of $6$, there is only $1$ balanced ballot path whose semisymmetric height is at most $2$: the path $P = (\vec{e}_1, \vec{e}_2, \vec{e}_3, \vec{e}_1, \vec{e}_2, \vec{e}_3)$. We obtain this fact because all other paths have at least two semisymmetric up-steps before their first semisymmetric down-step. Therefore, $\widehat{C}_{3, 2, 2}=1$.
\end{example}

The $k$-dimensional balanced ballot paths with semisymmetric height at most $u$ can be visualized as balanced ballot paths whose intermediate points are between the hyperplanes $$ x_1 = \frac{-(k-3)x_2 + \cdots + (1-k)x_k}{k-1} \mbox{ and } x_1 = \frac{-(k-3)x_2 + \cdots + (1-k)x_k + u}{k-1}. $$

Having defined the semisymmetric height, we introduce a new weight on $k$-dimensional balanced ballot paths. Since the weights of Dyck paths were defined as products of contributions from up-steps, we define them as products of contributions from steps. Unlike the weights of Dyck paths introduced in Section~\ref{section:intro}, which depend only on up-steps, the semisymmetric weight is the product of contributions from all steps. That way, our definition of semisymmetric weight generalizes variants of Dyck path weights that incorporate contributions from down-steps, such as those considered in An's thesis \cite{ansthesis}.

\begin{definition}[Semisymmetric Weight]
\label{defn:ssweight}  Let $P$ be a $k$-dimensional ballot path of length $kn$. Let $\vec{b} = (b_0, b_1, \dots)$ and $ \vec{c}=(c_0, c_1, \dots)$ be infinite sequences of integers. The \textbf{semisymmetric weight of $P$ with respect to $\vec{b}$ and $\vec{c}$} is $$sswt_{\vec{b}, \vec{c}}(P)= \left(\prod_{i=1}^{\lfloor \frac{k}{2} \rfloor n}b_{u_i}\right) \cdot \left(\prod_{j=1}^{(\lceil \frac{k}{2} \rceil) n}c_{u_{j}'}\right),$$ where $u_i$ is the semisymmetric height of the starting point of $i$th semisymmetric up-step of $P$ and $u_{j}'$ is the semisymmetric height of the resulting point of $j$th step of $P$ that is either a semisymmetric neutral step (in the case of odd $k$) or down-step.

For abbreviation, we define $$sswt_{\vec{b}}(P) = sswt_{\vec{b}, (1, 1, \dots)}(P).$$
\end{definition}

\begin{example}\label{ex:ComputeSSWT2}
Consider sequences of integers $\vec{b} = (b_0,b_1, \dots)$ and $\vec{c} = (c_0, c_1, \dots)$ along with the balanced ballot path $P = (\vec{e}_1, \vec{e}_2, \vec{e}_1, \vec{e}_3, \vec{e}_2, \vec{e}_3, \vec{e}_1, \vec{e}_2, \vec{e}_3)$ from Example~\ref{ex:3Dpathexample}. We compute the semisymmetric weight $sswt{\vec{b}, \vec{c}}(P)$ of the path. The semisymmetric up-steps start at semisymmetric heights $ 0$, $ 2$, and $0$, and contribute $b_0$, $b_2$, and $b_0$, respectively. The semisymmetric down-steps end at semisymmetric height $2$, $0$, and $0$, and contribute factors $c_2$, $c_0$, and $c_0$, respectively. All semisymmetric neutral steps start at a semisymmetric height of $2$, and each contributes a factor of $c_2$. This yields $sswt_{\vec{b}, \vec{c}}(P) = b_0 b_2 c_2^4c_0^2$.
\end{example}
\subsection{The \texorpdfstring{$k$}{k}-dimensional Semisymmetric Weighted Catalan Numbers (\texorpdfstring{$k$}{k}-dimensional SSWCNs)}\label{sec:MainObject}
Equipped with the notion of semisymmetric weight on balanced ballot paths, we are ready to introduce the two main objects of the paper: the $k$-dimensional semisymmetric weighted Catalan numbers ($k$-dimensional SSWCNs—this paper’s main $k$-dimensional generalization of weighted Catalan numbers—and the $k$-dimensional $u$-bounded semisymmetric weighted Catalan numbers ($k$-dimensional $u$-bounded SSWCNs).

We define the first main object of our paper, the $k$-dimensional semisymmetric weighted Catalan numbers (SSWCNs), as follows.

\begin{definition}\label{defn:kdimweighted}
Let $k$ and $n$ be positive integers. Let  $\vec{b} = (b_0, b_1, \dots)$ and $ \vec{c} = (c_0, c_1, \dots)$ be infinite sequences of integers. We define the $n$th \textbf{$k$-dimensional semisymmetric weighted Catalan number ($k$-dimensional SSWCN)} \[
\widehat{C}_{k,n}^{\vec{b}, \vec{c}} = \sum_P sswt_{\vec{b},\vec{c}}(P),
\] where the summation is over all $k$-dimensional balanced ballot paths of length $kn$. The quantity $sswt_{\vec{b}, \vec{c}}(P)$ is the semisymmetric weight defined in Definition~\ref{defn:ssweight}.
\end{definition}

In particular, we have $\widehat{C}_{k, n}^{(1, 1, \dots), (1, 1, \dots)} = C_{k, n}$, the $n$th $k$-dimensional Catalan number.

We use $\widehat{C}$ to denote $k$-dimensional SSWCNs and related sequences, distinguishing them from the $k$-dimensional weighted Catalan numbers, previously denoted by C in our earlier work \cite{MR5034574}. This notation distinguishes between these two generalizations and reinforces the difference in their underlying concepts of weight, as elaborated in Section~\ref{sec:ssheight}.

We introduce abbreviations for $k$-dimensional SSWCNs and semisymmetric weights.

\begin{definition}[Abbreviations for $\vec{c} = (1, 1, \dots)$]\label{def:shorthand1}
For sequences of integers $\vec{b} = (b_0, b_1, \dots)$, the \textbf{$n$th $k$-dimensional SSWCN with respect to $\vec{b}$} is $$\widehat{C}_{k,n}^{\vec{b}} = \widehat{C}_{k,n}^{\vec{b},(1,1,\dots)}.$$ For any sub-ballot path or balanced ballot path $P$, the \textbf{semisymmetric weight of $P$ with respect to $\vec{b}$} is  $$sswt_{\vec{b}}(P) = sswt_{\vec{b}, (1,1, \dots)}(P).$$
\end{definition}

In particular, observe that $\widehat{C}_{2, n}^{\vec{b}} = C_n^{\vec{b}}$, i.e., the weighted Catalan number $C_n^{\vec{b}}$, as defined in Section~\ref{sec:WCN}, is equal to the 2-dimensional SSWCN $\widehat{C}_{2, n}^{\vec{b}}$.

\begin{example}\label{ex:computation}
Suppose $\vec{b} = (b_0, b_1, \dots)$ and $\vec{c} = (c_0, c_1, \dots)$ are infinite sequences of integers. We compute $\widehat{C}^{\vec{b}, \vec{c}}{3, 2}$ by finding the semisymmetric weights of all 3-dimensional balanced ballot paths of length $6$, which we listed in Example~\ref{ex:EnumerationRunning}. Table~\ref{tab:ballotPathTablewWeights} lists these paths and their semisymmetric weights. From this data, we conclude:
$\widehat{C}^{\vec{b}, \vec{c}}{3, 2} =  b_0b_2c_4^2c_2c_0 + b_0b_2c_4c_2^2c_0 + b_0b_2c_4c_2^2c_0 + b_0^2c_2^2c_0^2 + b_0b_2c_2^3c_0.$

\begin{table}[H]
\centering
\begin{tabular}{|c|c|}
\hline  balanced ballot path & semisymmetric weight of balanced ballot path with respect to $(\vec{b}, \vec{c})$ \\ \hline
$(\vec{e}_1, \vec{e}_1, \vec{e}_2, \vec{e}_2, \vec{e}_3, \vec{e}_3)$ & $b_0b_2c_4^2c_2c_0$        \\ \hline $(\vec{e}_1, \vec{e}_2, \vec{e}_1, \vec{e}_2, \vec{e}_3, \vec{e}_3)$ & $b_0b_2c_4c_2^2c_0$ \\ \hline $(\vec{e}_1, \vec{e}_1, \vec{e}_2, \vec{e}_3, \vec{e}_2, \vec{e}_3)$ & $b_0b_2c_4c_2^2c_0$ \\ \hline $(\vec{e}_1, \vec{e}_2, \vec{e}_3, \vec{e}_1, \vec{e}_2, \vec{e}_3)$ & $b_0^2c_2^2c_0^2$ \\ \hline $(\vec{e}_1, \vec{e}_2, \vec{e}_1, \vec{e}_3, \vec{e}_2, \vec{e}_3)$ & $b_0b_2c_2^3c_0$ \\ \hline
\end{tabular}
\caption{3-dimensional ballot paths of length $6$ with respective weights.}
\label{tab:ballotPathTablewWeights}
\end{table}
\end{example}

In parallel with the $k$-dimensional SSWCNs, we also introduce a variant: the $k$-dimensional $u$-bounded SSWCNs. This forms the second main object of the paper.

\begin{definition}
\label{defn:k-dims-boundCat}
Let $n$ and $k$ be positive integers.  Let $\vec{b} = (b_0,b_1, \dots)$ and $\vec{c}=(c_0, c_1, \dots)$ be infinite sequences of integers. We define the $n$th \textbf{$k$-dimensional $u$-bounded SSWCN} to be $$\widehat{C}_{k, u, n}^{\vec{b}, \vec{c}} = \sum_{P}sswt_{\vec{b}, \vec{c}}(P),$$ where the sum is over all balanced ballot paths $P$ whose semisymmetric heights are at most $u$.
\end{definition}
\begin{definition}[Abbreviation for $\vec{c} = (1, 1, \dots)$]\label{def:abbrv2}For shorthand, the $n$th {$k$-dimensional $u$-bounded SSWCN with respect to $\vec{b}$} is $$\widehat{C}_{k, u, n}^{\vec{b}} = \widehat{C}_{k, u, n}^{\vec{b}, (1, 1, \dots)}.$$
\end{definition}
In particular, we have $\widehat{C}_{k, u, n} = \widehat{C}{k, u, n}^{(1, 1, \dots )}$.

We end the subsection with a table for notation for the $k$-dimensional SSWCNs and their variants.
\begin{table}[H]
\centering
\begin{tabular}{|c|c|c|}
\hline Notation & Description & Definition No. \\ \hline
$\widehat{C}_{k, u, n}$& the $n$th $k$-dimensional $u$-bounded semisymmetric Catalan number. &  Def.~\ref{def:bounded-and-height}. \\  \hline $\widehat{C}_{k, n}^{\vec{b}, \vec{c}}$ &  the $n$th $k$-dimensional SSWCN with respect to $\vec{b}$ and $\vec{c}$. & Def.~\ref{defn:kdimweighted}.\\  \hline $\widehat{C}_{k, n}^{\vec{b}}$ &  the $n$th $k$-dimensional SSWCN with respect to $\vec{b}$ ( and $\vec{c}  = (1, 1, \dots)$). & Def.~\ref{def:shorthand1}. \\ \hline $\widehat{C}_{k, u, n}^{\vec{b},  \vec{c}}$ &  the $n$th $k$-dimensional $u$-bounded SSWCN. & Def.~\ref{defn:k-dims-boundCat}. \\ \hline $\widehat{C}_{k, u, n}^{\vec{b}}$ &  the $n$th $k$-dimensional $u$-bounded SSWCN for $\vec{c} = (1, 1, \dots)$. & Def.~\ref{def:abbrv2}. \\ \hline
\end{tabular}
\caption{Notation table for $k$-dimensional SSWCNs and their variants.}
\label{tab:notations}
\end{table}

\subsection{Relation to Our Prior Generalization of Weighted Catalan Numbers}\label{sec:Relation}
In this section, we show that the $k$-dimensional SSWCNs are distinct from the $k$-dimensional weighted Catalan numbers from \cite{MR5034574}.

To clearly distinguish the two generalizations of weighted Catalan numbers, we begin by restating the definition of $k$-dimensional weighted Catalan numbers from \cite{MR5034574}. Given a sequence of integers $\vec{b} = (b_0, b_1, \dots)$, the \textbf{weight} $wt_{\vec{b}}(P)$ of a balanced ballot path $P$ is defined as $wt_{\vec{b}}(P) = b_{u_1}b_{u_2}\dots b_{u_n}$, where $u_i$ denotes the height of the starting point of the $i$th step, equal to $\vec{e}_1$. The \textbf{$n$th $k$-dimensional weighted Catalan number} is $C_{k, n}^{\vec{b}}= \sum_{P}wt_{\vec{b}}(P)$, where the sum ranges over all balanced ballot paths $P$ of $kn$ steps. Similarly, the \textbf{$n$th $k$-dimensional $u$-bounded weighted Catalan number} is $C_{k, u, n}^{\vec{b}}= \sum_{P}wt_{\vec{b}}(P)$, where the sum is restricted to all balanced ballot paths $P$ of $kn$ steps with $h_k(P) \leq u$.

\begin{example}
For the balanced ballot path $(\vec{e}_1, \vec{e}_1, \vec{e}_2, \vec{e}_2, \vec{e}_3, \vec{e}_3)$, we have $wt_{\vec{b}}((\vec{e}_1, \vec{e}_1, \vec{e}_2, \vec{e}_2, \vec{e}_3, \vec{e}_3)) = b_0b_2$. The steps in the $\vec{e}_1$ direction start at height $0$ and $2$. The $3$-dimensional weighted Catalan number is $C_{3, 2}^{\vec{b}} = 4b_0b_2 + b_0^2$. This is the sum of weights of all 3-dimensional balanced ballot paths of length $6$, which are listed in Example~\ref{ex:EnumerationRunning}.
\end{example}

Although both $k$-dimensional SSWCNs and weighted Catalan numbers are weighted sums over balanced ballot paths, the definitions of weight are fundamentally different. The semisymmetric weight is computed as a product of stepwise contributions along the entire path. In contrast, $wt_{\vec{b}}$ only uses contributions from steps that are equal to $\vec{e}_1$. Additionally, the semisymmetric weight relies on the semisymmetric height $g_k$, a concept distinct from the height $h_k$.

\begin{example}\label{ex:WhyNotTheSame}
Even for small values of $n$ and $k$, these differences affect the outcomes for $k$-dimensional SSWCNs and weighted Catalan numbers. To illustrate, we show that $C_{4, 1}^{\vec{b}} \neq \widehat{C}_{4, 1}^{\vec{b}}$. Consider the path $P = (\vec{e}_1, \vec{e}_2, \vec{e}_3, \vec{e}_4)$, which is the unique 4-dimensional balanced ballot path of length $4$. In this case, we find $\widehat{C}_{4, 4, 1}^{\vec{b}} = \widehat{C}_{4, 1}^{\vec{b}} = sswt_{\vec{b}, (1, 1, \dots)}(P) = b_0 b_3$, whereas $C_{4, 1}^{\vec{b}} = wt_{\vec{b}}(P) = b_0$.
\end{example}
The same two features, the use of the semisymmetric height $g_k$ and the fact that the semisymmetric weight incorporates contributions from all steps, also distinguish the $k$-dimensional $u$-bounded SSWCNs from the $k$-dimensional $u$-bounded weighted Catalan numbers in \cite{MR5034574}.
\subsection{Bounds on Semisymmetric Heights of \texorpdfstring{$k$}{k}-dimensional Balanced Ballot Paths}\label{sec:BasicObservations}

In this subsection, we determine the minimum and maximum values of the semisymmetric height for $k$-dimensional balanced ballot paths of nonzero length. These bounds serve two purposes in the paper. First, when we compute examples of $k$-dimensional $u$-bounded SSWCNs in Section~\ref{sec:kuboundedweightedCatalan}, they justify why we only consider $u \geq \left \lceil \frac{k}{2} \right\rceil \cdot \left \lfloor \frac{k}{2} \right\rfloor$. Second, these bounds are central in Section~\ref{sec:heighttriangles}, where we study the $k$-dimensional semisymmetric height triangle, which enumerates such paths according to their semisymmetric height.

We begin by finding the maximum semisymmetric height attainable by a $k$-dimensional balanced ballot path of length $kn$.

\begin{proposition}\label{prop:maxheight}
Let $n > 0$. Then, the largest possible semisymmetric height that is attainable by $k$-dimensional balanced ballot paths of length $kn$ is $\left\lfloor \frac{k}{2}\right\rfloor \left\lceil \frac{k}{2} \right\rceil n$.
\end{proposition}
\begin{proof}
First, we show that all $k$-dimensional balanced ballot paths of length $kn$ have a height of at most $\lfloor \frac{k}{2}\rfloor \lceil \frac{k}{2} \rceil n$. For each $i \in \{1,2, \dots, \lceil \frac{k}{2}\rceil\}$, there are $n$ steps of form $\vec{e}_i$. All other steps are semisymmetric down-steps. We find that the maximum semisymmetric height of any intermediate point in the path is no more than $\left \lfloor \frac{k}{2}\right\rfloor \lceil \frac{k}{2} \rceil n$.

Next, we show that there exists a $k$-dimensional balanced ballot path of $kn$ steps that has semisymmetric height $\left\lfloor \frac{k}{2} \right\rfloor \lceil \frac{k}{2} \rceil n$. Consider the balanced ballot path that first has $n$ consecutive steps in the $\vec{e}_1$ direction, then $n$ consecutive steps in the $\vec{e}_2$ direction, and so on. Immediately after the last step in the $\vec{e}_{\lceil \frac{k}{2} \rceil}$th direction, the path is at the intermediate point $\sum_{j=1}^{\lceil \frac{k}{2} \rceil}n \vec{e}_j$. The semisymmetric height at this point is $\sum_{j=1}^{\left\lceil \frac{k}{2}\right\rceil} (k-2j+1)n = n \lceil \frac{k}{2} \rceil \cdot \left\lfloor \frac{k}{2} \right\rfloor$. This fact completes the proof.\end{proof}

We now find the lower bound.

\begin{proposition}\label{prop:MinHeight}
Let $n \geq 1$. Then, the minimum possible semisymmetric height of a $k$-dimensional balanced ballot path of length $kn$ is $\left\lceil \frac{k}{2}\right\rceil \left\lfloor \frac{k}{2} \right\rfloor.$
\end{proposition}
\begin{proof}
We first show that any $k$-dimensional balanced ballot path $P = (\vec{s}_1, \vec{s}_2, \dots, \vec{s}_{kn})$ of $kn$ steps has a semisymmetric height of at least $\left\lceil \frac{k}{2}\right\rceil \left\lfloor \frac{k}{2} \right\rfloor.$  It suffices to show that $P$ has an intermediate point $\vec{x}$ with semisymmetric height at least $\left\lceil \frac{k}{2}\right\rceil \left \lfloor \frac{k}{2} \right\rfloor.$ We use $\vec{x} = (x_1, x_2, \dots, x_k) = \sum_{j=1}^{\left\lceil \frac{k}{2}\right\rceil}\vec{s}_{j}$ to denote the $\left\lfloor \frac{k}{2}\right\rfloor$th intermediate point of $P$. Because of the ballot property of balanced ballot paths (Definition~\ref{defn:k-dim-path}), we find that the first $\left\lfloor \frac{k}{2} \right\rfloor$ steps of $P$ are all semisymmetric up-steps (i.e., in ${\vec{e}_{1}, \vec{e}_2, \dots, \vec{e}_{\left\lfloor \frac{k}{2} \right\rfloor}}$); otherwise, there exist indices $i_{bad} \in \{1, 2, \dots, \left\lfloor \frac{k}{2} \right\rfloor \}$ and $i_{bad2} \in \{i_{bad}+1, i_{bad}+2, \dots, k\}$ such that both $x_{i_{bad}} = 0$ and $x_{i_{bad2}} > 0$, which contradicts the ballot property.  Because the first $\left\lfloor \frac{k}{2}\right\rfloor$ steps are all semisymmetric up-steps, we have $\sum_{j=1}^{\left\lfloor \frac{k}{2}\right\rfloor}x_j \geq \left\lfloor \frac{k}{2}\right\rfloor$. Thus, balanced ballot path $P$ has a semisymmetric height of at least $\lceil \frac{k}{2} \rceil \lfloor \frac{k}{2} \rfloor$.

Next, we show that there exists a $k$-dimensional ballot path of length $kn$ that has height $\lceil \frac{k}{2}\rceil \cdot \left\lfloor \frac{k}{2} \right\rfloor$. Consider the balanced ballot path $P_0 = (\vec{s}_1, \vec{s}_2, \dots, \vec{s}_{kn})$. It consists of $\vec{e}_1, \vec{e}_2, \dots, \vec{e}_k$ for the first $k$ steps, that same sequence for the next $k$ steps, and so on. The maximum height reached by an intermediate point in this path is $\sum_{j=1}^{\left\lfloor\frac{k}{2} \right\rfloor}(k-2j+1) =  \left\lceil \frac{k}{2} \right\rceil \cdot \left\lfloor\frac{k}{2} \right\rfloor$. This is reached at the intermediate point $\vec{e}_1+\vec{e}_2+\dots+\vec{e}_{\left\lfloor \frac{k}{2} \right\rfloor }$. Thus, by Definition~\ref{def:height-new}, the semisymmetric height of $P_0$ is $g_k(P_0) = \max \{g_k(\vec{0}+\sum_{j=1}^i\vec{s}_j):i \in \{0, 1, \dots, kn\}\} = \left\lceil \frac{k}{2} \right\rceil \cdot \left\lfloor\frac{k}{2} \right\rfloor$.

\end{proof}

One may be tempted to conjecture that for any $n \geq 1$ and $k \geq 2$, we have $\widehat{C}_{k, \left\lfloor\frac{k}{2} \right\rfloor \cdot \left\lceil\frac{k}{2} \right\rceil, n}=1$, i.e., there is exactly one $k$-dimensional balanced ballot path of $kn$ steps with semisymmetric height $\left\lfloor \frac{k}{2} \right\rfloor \cdot \left\lceil \frac{k}{2} \right\rceil$. For any positive $n$, there is only one Dyck path of length $2n$ with the smallest possible height; it has no intermediate points above the line $y=1$. However, this conjecture is false. Note the following counterexample.

\begin{example}
Consider $k=18$ and $n=2$. We show $\widehat{C}_{k, \left\lfloor\frac{k}{2} \right\rfloor \cdot \left\lceil\frac{k}{2} \right\rceil, n} > 1$, i.e., there are multiple $18$-dimensional balanced ballot paths of length $36$ and semisymmetric height of exactly $81$. One such path is $P_0 =(\vec{e}_1, \vec{e}_2, \dots, \vec{e}_{18}, \vec{e}_1, \vec{e}_2, \dots, \vec{e}_{18})$ from the proof of Proposition~\ref{prop:MinHeight}. We prove that there exists another path $P' \neq P_0$. First, consider the sub-ballot path $P{sub} = (\vec{e}_1,\vec{e}_2,\ldots,\vec{e}_{16},\vec{e}_1,\vec{e}_2)$ from $(0, 0, \dots, 0)$ to $(2, 2, 1, 1,\dots, 1,0, 0)$, which has a semisymmetric height of exactly $81$. The semisymmetric height of $(2, 2, 1, 1,\dots, 1,0, 0)$ is $g_{18}((2, 2, 1, 1,\dots, 1,0, 0)) = 64<81$. Next, append the sequence $(\vec{e}_{17}, \vec{e}_{17}, \vec{e}_{18}, \vec{e}_{18}, \vec{e}_3, \dots, \vec{e}_{16})$ to $P_{sub}$. The resulting sequence is a balanced ballot path $P'$ with $36$ steps and semisymmetric height $81$, as the sub-ballot path $(\vec{e}_{17}, \vec{e}_{17}, \vec{e}_{18}, \vec{e}_{18}, \vec{e}_3, \dots, \vec{e}_{16})$ from $(2, 2, 1, 1,\dots, 1,0, 0)$ to $(2, 2, \dots, 2)$ has a height of less than $81$.
\end{example}

\subsection{Auxiliary Definitions for Proofs}\label{sec:Auxiliary}
We introduce the $n$th $k$-dimensional $u$-bounded sub-SSWCNs: auxiliary variants of our $k$-dimensional SSWCNs. We use these auxiliary definitions only in the proofs of Theorem~\ref{thm:periodboundedkdim} in Sections~\ref{sec:periodicity} and of Theorems~\ref{thm:42},~\ref{thm:58}, and~\ref{thm:kboundedkdimensional} in Section \ref{sec:kuboundedweightedCatalan}. In particular, the proofs in those sections require decomposing balanced ballot paths at intermediate points, which motivates generalizing the notion of $k$-dimensional SSWCNs to the setting of sub-ballot paths.

We first introduce the notion of a semisymmetric weight for sub-ballot paths, which is the key tool in defining the $n$th $k$-dimensional $u$-bounded sub-SSWCNs.
\begin{definition}
\label{defn:weightsSSWTsubballot}
Let $\vec{a}$ and $ \vec{a}'$ both be vectors in $ \mathbb{Z}_{\geq 0}^k$. Let $P$ be a $k$-dimensional sub-ballot path $P$ from $\vec{a}$ to $\vec{a}'$. Let $\vec{b}= (b_0, b_1, \dots)$ and $ \vec{c} = (c_0, c_1, c_2, \dots, )$ be infinite sequences of integers. The \textbf{semisymmetric weight $sswt{\vec{b}, \vec{c}}(P)$ of $P$ with respect to $\vec{b}$ and $\vec{c}$} is the product $$sswt_{\vec{b}, \vec{c}}(P) = (b_{u_1}b_{u_2}\dots b_{u_{\ell}})(c_{u_1'}c_{u_2'}\dots c_{u_{\ell'}'}),$$ where the product is over all steps in $P$. In the product, we use $u_i$ to denote the semisymmetric height of the starting point of the $i$th semisymmetric up-step and  $u_{i}'$ to denote the semisymmetric height of the resulting point of the $i$th step of $P$ that is not a semisymmetric up-step.
\end{definition}

Using the semisymmetric weight of sub-ballot paths, we define the following generalization of $k$-dimensional, $u$-bounded SSWCNs, which we will use in the proof of Theorem~\ref{thm:periodboundedkdim}, which is the key result for Section~\ref{sec:periodicity}, and in the results for Section~\ref{sec:kuboundedweightedCatalan}
\begin{definition}\label{def:subSSWCN}
Let $\vec{b}$ and $\vec{c}$ be infinite sequences of integers. Let $\vec{a}$ be a vector in $\mathbb{Z}{\geq 0}^k$. Let $k, u$ be positive integers. Let $n $ be a nonnegative integer. The \textbf{$n$th $k$-dimensional $u$-bounded sub-SSWCN} for $\vec{a}$ is $$\widehat{C}_{k, u, \vec{a}, n}^{\vec{b}, \vec{c}} = \sum_{P}sswt_{\vec{b}, \vec{c}}(P),$$ with the summation being over all sub-ballot paths from $\vec{a}$ to $(n, \dots, n)$ whose semisymmetric heights never exceed $u$.
\end{definition}

As with the $k$-dimensional SSWCNs and their variants, we introduce an abbreviation of the mathematical expression for the $k$-dimensional $u$-bounded sub-SSWCNs for $\vec{c} = (1, 1, \dots)$. These abbreviations are only used in the proofs of Section~\ref{sec:kuboundedweightedCatalan}.

\begin{definition}[Abbreviation for $\vec{c} = (1, 1, \dots)$.]\label{def:abbrvSumSub}
We define the following abbreviation of \textbf{$n$th $k$-dimensional $u$-bounded sub-SSWCN}:
$$\widehat{C}_{k, u, \vec{a}, n}^{\vec{b}} = \widehat{C}_{k, u, \vec{a}, n}^{\vec{b}, (1, 1, \dots)}.$$
\end{definition}
In particular, we have $\widehat{C}_{k, u,n}^{\vec{b}, \vec{c}} = \widehat{C}_{k, u, \vec{0}, n}^{\vec{b}, \vec{c}}$

\section{Eventual Periodicity of \texorpdfstring{$k$}{k}-dimensional SSWCNs \texorpdfstring{$\pmod{m}$}{(mod m)}}\label{sec:periodicity}
Now that we have introduced new $k$-dimensional generalizations of weighted Catalan numbers, an arithmetic question arises: are these sequences eventually periodic modulo an integer $m$? In our previous paper \cite{MR5034574}, we proved eventual periodicity results for weighted $k$-dimensional Catalan numbers $C_{k,n}^{\vec{b}}$ (see Section~\ref{sec:Relation}). This leads us to ask whether the same phenomenon holds for the $k$-dimensional SSWCNs and their $u$-bounded variants modulo an integer $m$.

Our strategy is first to prove the eventual periodicity of $\widehat{C}_{k,u,n}^{\vec{b},\vec{c}} \pmod{m}$ by deriving linear recurrences for the $k$-dimensional $u$-bounded sub-SSWCNs defined in Section~\ref{def:subSSWCN}. We then apply this result to deduce eventual periodicity of the sequence $\{\widehat{C}_{k,n}^{\vec{b},\vec{c}} \pmod{m}\}_{n\geq 0}$ under suitable divisibility conditions on $\vec{b}$ and $\vec{c}$.

\begin{theorem}
\label{thm:periodboundedkdim}
    Let $m$ and $u$ be positive integers. Let $\vec{b}= (b_0, b_1, b_2, \dots)$ and $\vec{c}= (c_0, c_1, c_2, \dots)$ be infinite sequences of integers. The sequence $\{\widehat{C}_{k,u,n}^{\vec{b},\vec{c}}\pmod{m} \}_{n\geq 0}$ of the $k$-dimensional $u$-bounded SSWCNs modulo $m$ is  eventually periodic.
\end{theorem}
\begin{proof}
Before we begin the proof, we introduce some notation. Let $S$ be the collection of points $\vec{x}$ that have semisymmetric height at most $u$ and can be reached from $(0, 0, \dots, 0)$ by a sequence of $k$-step sub-ballot paths. In other words, $$S  = \left\{(w_1, w_2, \dots, w_k) \in \mathbb{Z}_{\geq 0}^k :  k \mid \sum_{j=1}^kw_j, w_1 \geq w_2 \geq \dots \geq w_k, \mbox{ and } g_k(\vec{w})\leq u\right\}.$$ Define $S' = \{(w_1, w_2, \dots, w_k) - w_k(1, 1, \dots , 1):(w_1, w_2, \dots, w_k) \in S \}$. In other words, set $S'$ is the subset of points $S$ whose $k$th coordinates are $0$. Let $\ell$ denote the size of $S'$. We know $\ell$ is finite because there are finitely many points $\vec{w} \in \mathbb{Z}_{\geq 0}^k$ whose last coordinate is zero and semisymmetric height is between $0$ and $u$ inclusive. 

For each $\vec{a} = (a_1, a_2, \dots, a_{k-1}, 0) \in S'$ and $n \in \{0,1,\dots,k\}$, let $\gamma_{\vec{a}, n-\sum_{j=1}^k a_k} = \widehat{C}_{k, u, \vec{a}, n}^{\vec{b},\vec{c}}$. Let $\vec{\gamma}_n$ be the vector $\vec{\gamma}_n = (\gamma_{\vec{a}, n})_{\vec{a} \in S}$. For $\vec{\gamma}_n$, the coordinate corresponding to $\vec{a} = \vec{0}$ is $1$, and all others are equal to $0$.

Our strategy is first to derive an $\ell \times \ell$ matrix $T$ such that $\vec{\gamma}_n = T \vec{\gamma}_{n-1}$ and then apply a pigeonhole argument to prove the eventual periodicity of the sequence of vectors $\{\vec{\gamma}_n \pmod m\}_{n\ge 0}$ in $(\mathbb{Z}/m\mathbb{Z})^\ell$. The argument based on the matrix-based recurrence and the pigeonhole principle is similar to that used in \cite{MR5034574} to prove the periodicity of $k$-dimensional $u$-bounded weighted Catalan numbers.
    
We first find the matrix $T$ for our matrix-based recurrence $\vec{\gamma}_n = T\vec{\gamma}_{n-1}$. First, observe that for any $\vec{a} \in S'$, we have the following recursive relation, \begin{equation}\label{eq:Recursion}\widehat{C}_{k, u, \vec{a}, n}^{\vec{b}, \vec{c}} = \sum_{\vec{a}^{\prime}} \sum_{P} \left(sswt_{\vec{b}, \vec{c}}(P) \right)\widehat{C}_{k, u, \vec{a}^{\prime}, n}^{\vec{b}, \vec{c}},\end{equation} where the outer sum is over all $\vec{a}^{\prime}=(a_1', a_2',\dots, a_k') \in S$ that can be reached from $\vec{a}$ via a sub-ballot path of $k$ steps and the inner sums over all sub-ballot paths $P$ of length $k$ from $\vec{a}$ to $\vec{a}^{\prime}$. 

We then observe that $\widehat{C}_{k, u, \vec{a}', n}^{\vec{b}, \vec{c}} = \widehat{C}_{k, u, \vec{a}'-(a'_k, a'_k, \dots, a'_k), n-a_k'}^{\vec{b}, \vec{c}}$. This holds because the set of sub-ballot paths from $\vec{a}'-(a'_k, a'_k, \dots, a'_k)$ to $(n-1, n-1, \dots, n-1)$ matches those from $\vec{a}'$ to $(n, n, \dots, n)$. Translation of all intermediate points by $(a_k', a_k', \dots, a_k')$ preserves both the ballot property and semisymmetric height. The height of $\vec{a}'-(a'_k, a'_k, \dots, a'_k)$ equals that of $\vec{a}'$ since $g_k((a_k', a_k', \dots, a_k')) = \sum_{i=1}^k(k+1-2i)a_k' = a_k'\sum_{i=1}^k(k+1-2i) = 0$. This equality and Equation (\ref{eq:Recursion}) together imply that for all $\vec{a}=(a_1, a_2, \dots, a_k) \in S'$, we have

\begin{equation}\label{eq:Second}\begin{split}
     & \gamma_{\vec{a}, n -\frac{\sum_{j=1}^k a_j}{k}} \\ & = \widehat{C}_{k, u, \vec{a}, n}^{\vec{b}, \vec{c}} \\ & = \sum_{\vec{a}^{\prime} = (a_1, a_2, \dots, a_k)} \sum_{P} \left(sswt_{\vec{b}, \vec{c}}(P) \right)\widehat{C}_{k, u, \vec{a}^{\prime} - (a_k',a_k', \dots, a_k'), n-a_k'}^{\vec{b}, \vec{c}}   \\ & =\sum_{\vec{a}^{\prime}= (a_1, a_2, \dots, a_k)} \left(\sum_{P} sswt_{\vec{b}, \vec{c}}(P) \right)\gamma_{\vec{a}^{\prime} - (a_k', a_k', \dots, a_k'), n -\frac{\sum_{j=1}^k a_j}{k}-1},\end{split}
\end{equation}

where the outer sum in each expression is over all $\vec{a}^{\prime}=(a_1', a_2',\dots, a_k') \in S$ reachable from $\vec{a}$ by a sub-ballot path of $k$ steps, and the inner sum is over all sub-ballot paths $P$ of length $k$ from $\vec{a}$ to $\vec{a}^{\prime}$. Because vector $\vec{a}'-(a_k', a_k', \dots, a_k')$ is in $S'$, Equation (\ref{eq:Second}) forms a linear relation between $\gamma_{\vec{a}, n}$ terms in $\{\gamma_{\vec{w}, n-1}: \vec{w} \in S'\}$. Here, the coefficient of $\gamma_{\vec{w}, n-1}$ is $(\sum_{P}sswt_{\vec{b}, \vec{c}}(P))$, where the sum is over all $P$ from $\vec{a}$ to $\vec{a}'$.

Thus, if we take Equation (\ref{eq:Second}) for all vectors $\vec{a} \in S'$ and find the coefficients of $\gamma_{\vec{w}, n-1}$ for each $w \in S'$ in that equality, we obtain a finite matrix $T$ such that $\vec{\gamma}_n = T \vec{\gamma}_{n-1}$. The matrix is finite because $|S'| < \infty$.  

Consider left multiplication by $T$ as the function $f:(\mathbb{Z}/m\mathbb{Z})^{\ell} \to (\mathbb{Z}/m\mathbb{Z})^{\ell}$. Since $f$ is an endomorphism and $(\mathbb{Z}/m\mathbb{Z})^{\ell}$ has finitely many values, the pigeonhole principle applies. It tells us there are integers $t, \omega$ so that $\vec{\gamma}_t \equiv  \vec{\gamma}_{t+\omega} \pmod{m}$. In other words, we have $\widehat{C}_{k, u,\vec{a}, t+t'}^{\vec{b}, \vec{c}} \equiv \widehat{C}_{k, u,\vec{a}, t}^{\vec{b}, \vec{c}} \pmod{m}$. Thus, for all $n \geq t$, $t' \geq 0$, and $\vec{a} \in S'$, $\widehat{C}_{k, u,\vec{a}, n}^{\vec{b}, \vec{c}} \equiv \widehat{C}_{k, u,\vec{a}, n+t'\omega}^{\vec{b}, \vec{c}} \pmod{m}$. Since $\vec{0} \in S'$ and $\widehat{C}_{k, u, \vec{0},n}^{\vec{b}, \vec{c}} = \widehat{C}_{k, u, n}^{\vec{b}, \vec{c}}$, we conclude $\widehat{C}_{k, u, n}^{\vec{b}, \vec{c}} \equiv \widehat{C}_{k, u, n+t'\omega}^{\vec{b}, \vec{c}} \pmod{m}$ for all $n \geq t$ and $t' \geq 0$. This completes the proof.
\end{proof}
    \begin{example}
We illustrate the argument of Theorem~\ref{thm:periodboundedkdim} in the case \(k=3\) and \(u=5\). In this case, we have $
S=\{(d,d,d): d\in\mathbb{Z}_{\geq 0}\}\cup\{(d+2,d+1,d): d\in\mathbb{Z}_{\geq 0}\}$
and $S'=\{(0,0,0),(2,1,0)\}$. Define $
\vec{\gamma}_n=\bigl(\gamma_{(0,0,0),n},\gamma_{(2,1,0),n-1}\bigr)$, where $\gamma_{(0,0,0),n} = \widehat{C}_{3,5,(0,0,0),n}^{\vec b,\vec c} = \widehat{C}_{3,5,n}^{\vec b,\vec c}$
and $\gamma_{(2,1,0),n-1} =\widehat{C}_{3,5,(2,1,0),n}^{\vec b,\vec c}$. Applying Equation~(\ref{eq:Second}) to each vector in \(S'\), we obtain the equations
$
\gamma_{(0,0,0),n}
= b_0c_2c_0\,\gamma_{(0,0,0),n-1} +(b_0b_2c_4+b_0c_2b_2)\,\gamma_{(2,1,0),n-1}$ and 
and $\gamma_{(2,1,0),n-1} =(c_4c_2c_0+c_2^2c_0)\,\gamma_{(0,0,0),n-2} +
(2c_4c_2b_2+c_2^2b_2)\,\gamma_{(2,1,0),n-2}$. Thus, we obtain the matrix $T= \begin{bmatrix}
b_0c_2c_0 & b_0b_2c_4+b_0c_2b_2 \\
c_4c_2c_0+c_2^2c_0 & 2c_4c_2b_2+c_2^2b_2
\end{bmatrix}$ and the linear relation $\vec{\gamma}_n=T\vec{\gamma}_{n-1}$.
Because left multiplication by \(T\) defines a function from the finite set \((\mathbb{Z}/m\mathbb{Z})^2\) to itself, the pigeonhole principle implies that there exist integers \(t\) and \(\omega\) such that $\vec{\gamma}_t\equiv \vec{\gamma}_{t+\omega}\pmod{m}$. It follows that the sequence \(\{\vec{\gamma}_n \bmod m\}_{n\geq 0}\) is eventually periodic. Consequently, the sequence $\{\gamma_{(0,0,0),n}\bmod m\}_{n\geq 0} = \{\widehat{C}_{3,5,n}^{\vec b,\vec c}\bmod m\}_{n\geq 0}$ is eventually periodic as well.
\end{example}

We now derive the corresponding result for the $k$-dimensional SSWCNs, whose proof follows from Theorem~\ref{thm:periodboundedkdim}. This result is the analogue of Theorem 3.2 of \cite{MR5034574}, which states that if there exists an integer $u$ such that $b_u, b_{u+1}, \dots, b_{u+k-1}$ are all divisible by $m$, then the sequence $\{C_{k,n}^{\vec{b}} \pmod{m}\}_{n \geq 0}$ is eventually periodic.

\begin{theorem}\label{thm:period-weighted-kdim}
Let $m$ be a positive integer, and let $\vec b=(b_0,b_1,\dots)$ and $\vec c=(c_0,c_1,\dots)$ be infinite sequences of integers. Suppose there exists a positive integer $u$ such that at least one of the following holds:
\begin{enumerate}
    \item each element of $\{b_u,b_{u+1},\dots,b_{u+k-1}\}$ is divisible by $m$; or
    \item each element of $\{c_{u-1},c_u,\dots,c_{u+k-2}\}$ is divisible by $m$.
\end{enumerate}
Then the sequence $\{\widehat{C}^{\vec{b},\vec{c}}_{k,n} \pmod{m}\}_{n\ge 0}$ is eventually periodic.
\end{theorem}
\begin{proof}
The idea of the proof is to show that balanced ballot paths $P$ with sufficiently large semisymmetric height contribute $0 \pmod{m}$ to $\widehat{C}_{k, n}^{\vec{b}, \vec{c}} \pmod{m}$. This computation reduces the problem to the one on $k$-dimensional $u$-bounded SSWCNs, where Theorem~\ref{thm:periodboundedkdim} applies. 

We show that any balanced ballot path with semisymmetric height of more than $u+k-2$ has a semisymmetric weight divisible by $m$. We consider the two cases of the hypothesis separately. 

First, suppose $u$ is such that the integers $b_u, b_{u+1}, \dots, b_{u+k-1}$ are all divisible by $m$. Consider any $k$-dimensional balanced ballot path $P$ with semisymmetric height greater than $u+k-2$. There must be a semisymmetric up-step of $P$ starting at an intermediate point $\vec{x}$, where the semisymmetric height $g_k(\vec{x})$ is between $u$ and $u+k-2$. Semisymmetric up-steps increase the semisymmetric height of their endpoints by at least $1$ and at most $k-1$. By the definition of semisymmetric weight, this step contributes a factor $b_i$ for some $i\in\{u\dots,u+k-1\}$. Thus, the semisymmetric weight of $P$ is a multiple of some element in $\{b_u, b_{u+1}, \dots, b_{u+k-1}\}$. Therefore, for any $k$-dimensional balanced ballot path $P$ of length $kn$ with semisymmetric height greater than $u+k-2$, the semisymmetric weight of $P$ is divisible by $m$, i.e., $sswt_{\vec{b}, \vec{c}}(P) \equiv 0 \pmod{m}$.

Next, consider the case where $u$ is such that $c_{u+k-2}, c_{u+k-3}, \dots, c_{u-1}$ are all divisible by $m$. Let $P$ be any $k$-dimensional balanced ballot path reaching semisymmetric height greater than $u+k-2$. There must be some step that is not an up-step and has a resulting intermediate point $\vec{y}$ with semisymmetric height $g_k(\vec{y}) \in \{u+k-2, u+k-3, \dots, u-1\}$. This is because steps that are not up-steps (i.e., down-steps and neutral steps) weakly decrease the semisymmetric heights of their resulting intermediate points by at least $0$ and by at most $k-1$. Thus, there exists a semisymmetric down-step or neutral step ending at a point whose semisymmetric height lies in $\{u+k-2, u+k-3, \dots, u-1\}$, and hence $sswt_{\vec{b}, \vec{c}}(P)$ has a factor in $\{c_{u+k-2}, c_{u+k-3}, \dots, c_{u-1}\}$. This implies $sswt_{\vec{b}, \vec{c}}(P)$ is divisible by $m$. 

Thus, any $k$-dimensional balanced ballot path $P$ with semisymmetric height greater than $u+k-2$ yields $sswt_{\vec{b}, \vec{c}}(P) \equiv 0 \pmod{m}$. Consequently, we obtain the congruence $\widehat{C}_{k, n}^{\vec{b}, \vec{c}} = \sum_{P}sswt_{\vec{b}, \vec{c}}(P) \equiv \sum_{P'} sswt_{\vec{b}, \vec{c}}(P') \pmod{m} \equiv \widehat{C}_{k, u+k-2, n}^{\vec{b}, \vec{c}} \pmod{m}$, where the summation is over all $k$-dimensional balanced ballot paths $P$ of length $kn$ with semisymmetric height at most $u+k-2$. Since the sequence $\{\widehat{C}_{k,u+k-2,n}^{\vec{b},\vec{c}} \pmod m\}_{n \geq 0}$ is eventually periodic by Theorem~\ref{thm:periodboundedkdim} and $\widehat{C}_{k,n}^{\vec{b},\vec{c}} \equiv \widehat{C}_{k,u+k-2,n}^{\vec{b},\vec{c}} \pmod m$, the sequence $\{\widehat{C}_{k,n}^{\vec{b},\vec{c}} \pmod m\}_{n \geq 0}$ is eventually periodic.
\end{proof}

We now derive a sufficient condition for the eventual periodicity of $\widehat{C}_{k,n}^{\vec{b},\vec{c}}$ modulo $m$ when $m$ divides certain products of two entries of $\vec{b}$. As in Theorem~\ref{thm:period-weighted-kdim}, the proof of the following result also follows from Theorem~\ref{thm:periodboundedkdim}.

\begin{theorem}\label{thm:Productperiodicity}
Let $m$ be a positive integer, and let $\vec b=(b_0,b_1,\dots)$ and $\vec c=(c_0,c_1,\dots)$ be infinite sequences of integers. Suppose there exists a positive integer $u$ such that
\[
b_j b_{j'} \equiv 0 \pmod m
\]
for every pair of distinct integers $j,j' \in \{u,u+1,\dots,u+2k-1\}$. Then the sequence $\{\widehat{C}^{\vec{b},\vec{c}}_{k,n} \pmod{m}\}_{n\ge 0}$ is eventually periodic.
\end{theorem}
\begin{proof}
The gist of this proof is to show that balanced ballot paths $P$ with semisymmetric height exceeding $u+2k-1$ contribute $0 \pmod{m}$ to $\widehat{C}_{k, n}^{\vec{b}, \vec{c}}$. This computation reduces the problem to the one on $k$-dimensional $u$-bounded SSWCNs, where Theorem~\ref{thm:period-weighted-kdim} applies. 

 Let $P$ be any $k$-dimensional balanced ballot path with semisymmetric height larger than $u+2k-1$. We show that the semisymmetric weight $sswt_{\vec{b}, \vec{c}}(P)$ is divisible by $m$. First, we observe that path $P$ must contain at least two up-steps whose starting heights lie in $\{u, u+1, \ldots, u+k-1\}$. In particular, there exist semisymmetric up-steps starting at heights $u_1 \in \{u,u+1,\ldots,u+k-1\}$ and $u_2 \in \{u,u+1,\ldots,u+2k-1\}$. Thus, the semisymmetric weight $sswt_{\vec{b}, \vec{c}}(P)$ of the balanced ballot path $P$ is divisible by $b_{u_1}b_{u_2}$. Consequently, it is divisible by $m$. 

Since $\widehat{C}_{k, n}^{\vec{b}, \vec{c}} = \sum_{P}sswt_{\vec{b}, \vec{c}}(P) \equiv  \sum_{P'}sswt_{\vec{b}, \vec{c}}(P')  \pmod{m} \equiv \widehat{C}_{k, u+2k-1, n}^{\vec{b}, \vec{c}} \pmod{m} $, with the first sum being over all $k$-dimensional balanced ballot paths of length $kn$ and the latter being over the subset of such paths with semisymmetric height at most $u+2k-1$.  Since $\widehat{C}_{k, n}^{\vec{b}, \vec{c}} \equiv  \widehat{C}_{k, u+2k-1, n}^{\vec{b}, \vec{c}} \pmod{m} $  and sequence $\{  \widehat{C}^{\vec{b}}_{k, u+2k-1, 1}  \pmod{m} \}_{n\geq 0}$ is eventually periodic by Theorem \ref{thm:periodboundedkdim}, we conclude that sequence $\{\widehat{C}_{k, n}^{\vec{b}, \vec{c}} \pmod{m}\}_{n \geq 0}$ is eventually periodic.
\end{proof}

\section{Examples of \texorpdfstring{$k$}{k}-dimensional \texorpdfstring{$u$}{u}-bounded SSWCNs}\label{sec:kuboundedweightedCatalan}
To complement our general periodicity results from Section~\ref{sec:periodicity}, we derive formulas for $k$-dimensional $u$-bounded SSWCNs $\widehat{C}_{k,u,n}^{\vec{b}}$ for select small values of $u$ and $k$. As corollaries, we compute closed formulas for the corresponding $k$-dimensional $u$-bounded semisymmetric Catalan numbers $\widehat{C}_{k,u,n}$ (Definition~\ref{def:bounded-and-height}). These formulas also connect the study of $k$-dimensional $u$-bounded SSWCNs to the broader theory of integer sequences (e.g., Corollary~\ref{cor:OEISEquals}).

We restrict our attention to $\widehat{C}_{k,u,n}^{\vec b}$ (rather than $\widehat{C}_{k,u,n}^{\vec b,\vec c}$) and to small values of $k$ and $u$, where explicit formulas can be computed. For each fixed $k$ considered in this section, we begin with a proposition treating the smallest admissible value of $u$ and then proceed to compute  $\widehat{C}_{k,u,n}^{\vec b}$ for the third smallest admissible value.

For the theorems of this section (Theorems~\ref{thm:kboundedkdimensional},~\ref{thm:42}, and~\ref{thm:58}), we employ the following general strategy. We decompose $k$-dimensional balanced ballot paths into smaller $k$-step sub-ballot paths and use this decomposition to derive recursive relations. More specifically, we construct linear relations between $\widehat{C}_{k,u,n}^{\vec{b}}$ and the $k$-dimensional $u$-bounded sub-SSWCNs $\widehat{C}_{k,u,\vec{a},n}^{\vec{b}}$ introduced in Section~\ref{sec:Auxiliary} (Definition~\ref{def:subSSWCN}). The coefficients in these relations are given by weighted sums over sub-ballot paths of length $k$. These results are similar in spirit to the examples of $k$-dimensional weighted Catalan numbers computed in \cite{MR5034574}.

For $k=3$, Proposition~\ref{prop:MinHeight} implies that every nonempty $3$-dimensional balanced ballot path of length $3n$ has semisymmetric height at least $2$. Thus, $u=2$ is the smallest admissible value of $u$ we consider for $k=3$, and we begin by determining $\widehat{C}_{3, 2, n}^{\vec{b}}$ and $\widehat{C}_{3, 3, n}^{\vec{b}}$.
\begin{proposition}\label{prop:three}
    For any integer $n \geq 1$, we have \[\widehat{C}_{3,2,n}^{\vec b}=\widehat{C}_{3,3,n}^{\vec b}=(b_0)^n.\]
\end{proposition}
\begin{proof}
The sequence formed by repeating $(\vec e_1, \vec e_2, \vec e_3)$ $n$ times is the unique nonempty 3-dimensional balanced ballot path of length $3n$ with semisymmetric height of at most 2.

Indeed, the first step of any nonempty $3$-dimensional balanced ballot path must be $\vec {e} _ {1} $. If the second step were $\vec e_1$, then the semisymmetric height would become $4$; if it were $\vec e_3$, the ballot condition would fail. Thus, the second step must be $\vec {e} _2$. Similarly, the third step must be $\vec {e} _3$. Removing this initial block and repeating the same argument shows that the only such path is $P_0$.

Moreover, by Proposition~\ref{prop:MinHeight}, the path $P_0$ has semisymmetric height $2$. Furthermore, each semisymmetric up-step of $P_0$ starts at semisymmetric height $0$. Thus, we obtain 
\[
\widehat{C}_{3,2,n}^{\vec b}=\widehat{C}_{3,3,n}^{\vec b}=sswt_{\vec b}(P_0)=b_0^n.
\]\end{proof}
We now turn to proving a formula for the $3$-dimensional, $4$-bounded SSWCN $\widehat{C}_{3,4,n}^{\vec b}$.
\begin{theorem}
\label{thm:kboundedkdimensional}
    Let $\vec{b} = (b_0, b_1, \dots)$ be an infinite sequence of integers, and let $n \geq 2$. Then, the following recursive relation holds:
    \[\widehat{C}_{3, 4, n}^{\vec{b}} = (b_0+3b_2) \widehat{C}_{3, 4, n-1}^{\vec{b}} + b_0b_2\widehat{C}_{3, 4, n-2}^{\vec{b}}.\]
\end{theorem}

\begin{proof}
    For brevity, let $A_n = \widehat{C}_{3, 4,n}^{\vec{b}}$. Let  $B_{n-1}$ denote $\widehat{C}_{3, 4, (2, 1, 0), n}^{\vec{b}}$, the sum of semisymmetric weights of sub-ballot paths $P$ from $(2,1,0)$ to $(n,n,n)$ such that $g_k(P) \leq 4$. We prove the recursive formulas $A_n = b_0 A_{n-1}+2b_0b_2B_{n-1}$ and $B_{n-1} = 2A_{n-2}+3b_2 B_{n-2}$ for all $n \geq 2$  by considering the sub-ballot paths of 3 steps that we can append together to form a full 3-dimensional balanced ballot path of length $3n$. Then, we solve for $A_n$ to obtain the expression $\widehat{C}_{3, 4, n}^{\vec{b}} = (b_0+3b_2) \widehat{C}_{3, 4, n-1}^{\vec{b}} + b_0b_2\widehat{C}_{3, 4, n-2}^{\vec{b}}$.

    We first prove $A_n = b_0A_{n-1}+2b_0b_2B_{n-1}$. There is only one 3-dimensional sub-ballot path of $3$ steps from $(0, 0, 0)$ to $(1, 1, 1)$ that does not exceed semisymmetric height $4$. This sub-ballot path has semisymmetric weight $b_0$. Also, the 3-dimensional 4-bounded SSWCN $\widehat{C}_{3, 4, (1, 1, 1), n}^{\vec{b}}$, the sum of semisymmetric weights of sub-ballot paths from $(1, 1, 1)$ to $(n, n, n)$ with semisymmetric height at most $4$, equals $\widehat{C}_{3, 4, n}^{\vec{b}}$, the sum of semisymmetric weights of sub-ballot paths from $(0, 0, 0)$ to $(n-1, n-1, n-1)$ with semisymmetric height at most $4$. Thus, we obtain the summand $b_0\widehat{C}_{3, 4, n-1}^{\vec{b}} = b_0A_{n-1}$. Then, note that there are only $2$ sub-ballot paths of $3$ steps from $(0, 0, 0)$ to $(2, 1, 0)$ not exceeding the semisymmetric height of $4$: these sub-ballot paths are $(\vec{e}_1, \vec{e}_1, \vec{e}_2)$ and $(\vec{e}_1, \vec{e}_2, \vec{e}_1)$. Both paths have semisymmetric weight $b_0b_2$. Also, the value of the $3$-dimensional $4$-bounded sub-SSWCN $\widehat{C}_{3, 4, (2, 1, 0), n}^{\vec{b}} = B_{n-1}$ is the sum of semisymmetric weights of sub-ballot paths from $(2, 1, 0)$ to $(n, n, n)$ that have semisymmetric height at most $4$. Thus, we obtain the summand $2b_0b_2B_{n-1}$. In sum, we get $A_n = b_0A_{n-1}+2b_0b_2B_{n-1}$.
    
    Now, we prove that $B_{n-1} = 2A_{n-2}+3b_2B_{n-2}$. Observe that are only two sub-ballot paths of $3$ steps from $(2, 1, 0)$ to $(2, 2, 2)$ that have semisymmetric height at most $4$. These sub-ballot paths are $(\vec{e}_2, \vec{e}_3, \vec{e}_3)$ and $(\vec{e}_3, \vec{e}_2, \vec{e}_3)$. Both have a semisymmetric weight of $1$. Also, observe that the $3$-dimensional $4$-bounded sub-SSWCN $\widehat{C}_{3, 4, (2, 2, 2), n}^{\vec{b}}$, the sum of semisymmetric weights of sub-ballot paths from $(2, 2,2)$ to $(n, n, n)$ with semisymmetric height at most $4$, equals $\widehat{C}_{3, 4, n-2}^{\vec{b}} = A_{n-2}$. Thus, we obtain the summand $2A_{n-2}$. Then, note that there are only three 3-step sub-ballot paths from $(2, 1, 0)$ to $(3, 2, 1)$ with a semisymmetric height of at most $4$; these sequences are $(\vec{e}_2, \vec{e}_3, \vec{e}_1), (\vec{e}_3, \vec{e}_2, \vec{e}_1),$ and $(\vec{e}_3, \vec{e}_1, \vec{e}_2)$, which all have semisymmetric weight $b_2$. Since we have $\widehat{C}_{3,4, (3, 2, 1)}^{\vec{b}} = B_{n-2}$ by translation by $(1, 1, 1)$, we obtain the summand $3b_2B_{n-2}$. In sum, we get $B_{n-1} = 2A_{n-2}+3b_2 B_{n-2}$.

 From $A_n = b_0A_{n-1}+2b_0b_2B_{n-1}$, we obtain $B_{n-1} = \frac{A_n- b_0A_{n-1}}{2b_0b_2}$. Substituting into $B_{n-1} = 2A_{n-2}+3b_2B_{n-2}$, we obtain $B_{n-1}= 2A_{n-2} + (3b_2) \frac{A_{n-1}- b_0A_{n-2}}{2b_0b_2} = \frac{3}{2b_0}A_{n-1} +\frac{1}{2}A_{n-2}$. Substituting the last identity into $A_n = b_0A_{n-1}+2b_0b_2B_{n-1}$, we obtain $ A_n = (b_0+3b_2) A_{n-1} + b_0b_2A_{n-2}$. This completes the proof.
\end{proof}

We observe $\widehat{C}_{3, 4, 0}^{\vec{b}} = 1$, since there is one and only one $3$-dimensional balanced ballot path of length $0$. Also, we have $\widehat{C}_{3, 4, 1}^{\vec{b}} = b_0$, since path $(\vec{e}_1, \vec{e}_2, \vec{e}_3)$ is the sole 3-dimensional balanced ballot path of length $3$ with semisymmetric height at most $4$. By applying the characteristic root technique for finding formulas for integer sequences with linear recurrences (c.f., Chapter 4.4 of Levin's reference \cite{levin2025discrete}), we obtain the following formula for the case when $b_0^2+10b_0b_2+9b_2^2 \neq 0$:
 \begin{equation}\label{eq:Equation1}\begin{split}
     \widehat{C}_{3, 4, n}^{\vec{b}} = &   \left(\frac{b_0-3b_2+\sqrt{b_0^2+10b_0b_2+9b_2^2}}{2\sqrt{b_0^2+10b_0b_2+9b_2^2}}\right)  \left(\frac{b_0+3b_2 + \sqrt{b_0^2+10b_0b_2+9b_2^2}}{2}\right)^n \\ & + \left(\frac{-b_0+3b_2+\sqrt{b_0^2+10b_0b_2+9b_2^2}}{2\sqrt{b_0^2+10b_0b_2+9b_2^2}}\right) \left(\frac{b_0+3b_2 - \sqrt{b_0^2+10b_0b_2+9b_2^2}}{2}\right)^n\end{split}
 \end{equation}

In the following corollary of Theorem~\ref{thm:kboundedkdimensional}, we obtain that $\widehat{C}_{3, 4, n}$, the number of $3$-dimensional balanced ballot paths of length $3n$ with semisymmetric height at most $4$, is the $n$th term of the OEIS sequence A015448 \cite{oeis}.

\begin{corollary}\label{cor:OEISEquals}
    Let $n$ be any nonnegative integer. Then, we have $$\widehat{C}_{3, 4, n} = \frac{1}{2\sqrt{5}}\left( (1+\sqrt{5})(2-\sqrt{5})^n - (1-\sqrt{5})(2+\sqrt{5})^n\right).$$ 
\end{corollary}
\begin{proof}
We observe that $\widehat{C}_{3, 4, n} = \widehat{C}_{3, 4, n}^{(1, 1, \dots)}$. We evaluate Equation (\ref{eq:Equation1}) for $\vec{b} = (1, 1, \dots)$ and obtain $\frac{1}{2\sqrt{5}}\left( (1+\sqrt{5})(2-\sqrt{5})^n - (1-\sqrt{5})(2+\sqrt{5})^n\right)$.
\end{proof}

We now compute examples of $\widehat{C}_{4, u, n}^{\vec{b}}$. Proposition~\ref{prop:MinHeight} implies that every nonempty $4$-dimensional balanced ballot path of length $4n$ has semisymmetric height at least $4$. We first find the closed formulas for $\widehat{C}_{4, 4, n}^{\vec{b}}$ and $\widehat{C}_{4, 5, n}^{\vec{b}}$ via a proof similar to the one from Proposition~\ref{prop:three}.
\begin{proposition}
    For any integer $n \geq 1$, we have $\widehat{C}_{4, 4, n}^{\vec{b}} = \widehat{C}_{4, 5, n}^{\vec{b}} = (b_0b_3)^n$.
\end{proposition}
\begin{proof}
    Observe that the path formed by repeating $(\vec{e}_1, \vec{e}_2, \vec{e}_3, \vec{e}_4)$ $n$ times is the unique nonempty 4-dimensional balanced ballot path of length $4n$ with semisymmetric height of at most $5$. We show this by applying the same argument as the one used in Proposition~\ref{prop:three} for the sequence $(\vec{e}_1, \vec{e}_2, \vec{e}_3)$ repeated $n$ times. 
    
    First, observe that the first step of any nonempty $4$-dimensional balanced ballot path must be $\vec {e}_{1} $. If the second step is $\vec{e}_1$, then the semisymmetric height of the resulting intermediate point would become $6$; if it were $\vec {e} _3$ or $\vec{e}_4$, the ballot property from Definition~\ref{defn:k-dim-path} would fail. Thus, the second step must be $\vec{e}_2$. Likewise, the third step must be $\vec {e}_3$. Otherwise, if the second step is $\vec{e}_1$, then the semisymmetric height of the resulting intermediate point would become $7$; if it were $\vec{e}_2$ or $\vec{e}_4$, the ballot property from Definition~\ref{defn:k-dim-path} would fail. Similarly, the fourth step must be $\vec{e}_4$. Removing this initial instance of $(\vec{e}_1, \vec{e}_2, \vec{e}_3, \vec{e}_4)$ and repeating the same argument shows that the only path with semisymmetric height at most $5$ is $P_0$.

Moreover, by Proposition~\ref{prop:MinHeight}, the path $P_0$ has semisymmetric height $4$. Furthermore, it has $n$ semisymmetric up-steps starting at intermediate points of semisymmetric height $0$ and an additional $n$ semisymmetric up-steps starting at intermediate points of semisymmetric height $3$. Therefore, we have $\widehat{C}_{4,4,n}^{\vec b}=\widehat{C}_{4,5,n}^{\vec b}=sswt_{\vec b}(P_0)=(b_0b_3)^n$.
\end{proof}

Having determined $\widehat{C}_{4,4,n}^{\vec b}$ and $\widehat{C}_{4,5,n}^{\vec b}$, we next derive the closed form for $\widehat{C}_{4,6,n}^{\vec b}$.

\begin{theorem}
\label{thm:42}
    Let $\vec{b} = (b_0, b_1, \dots)$ be an infinite sequence of integers. For $n \geq 1$, we have \[\widehat{C}_{4, 6,n}^{\vec{b}} = b_0b_3(2b_0b_3)^{n-1}.\]
    \end{theorem}
    
    \begin{proof}
    
        Let $A_n = \widehat{C}_{4, 6, n}^{\vec{b}}$. Let $B_{n-1} $ be $\widehat{C}_{4, 6, (2, 1, 1, 0), n}^{\vec{b}}$, the sum of semisymmetric weights $sswt_{\vec{b}}(P)$ over sub-ballot paths $P$ from $(2, 1, 1, 0)$ to $(n, n, n,n)$ that do not exceed semisymmetric height $6$. We prove the two linear recurrences $A_n = b_0b_3A_{n-1}+b_0b_3^2B_{n-1}$ and $B_{n-1} = b_3A_{n-2}+b_3^2B_{n-2}$ for $n \geq 2$.

To begin, there is only one sub-ballot path from $(0, 0, 0, 0)$ to $(1,1,1,1)$: this path is $(\vec{e}_1, \vec{e}_2, \vec{e}_3, \vec{e}_4)$, which has semisymmetric weight $b_0b_3$. Hence, we obtain the summand $b_0b_3A_{n-1}$. Also, there is only one sub-ballot path from $(0, 0, 0, 0)$ to $(2,1,1,0)$: this path is $(\vec{e}_1, \vec{e}_2, \vec{e}_3, \vec{e}_1)$, which has semisymmetric weight $b_0b_3^2$. Hence, we obtain the summand $b_0b_3^2B_{n-1}$. Thus, we obtain $A_n = b_0b_3A_{n-1}+b_0b_3^2B_{n-1}$.

Next, there is exactly one sub-ballot path from $(2, 1, 1, 0)$ to $(2, 2, 2, 2)$: this path is $(\vec{e}_4, \vec{e}_2, \vec{e}_3, \vec{e}_4)$, which has semisymmetric weight $b_3$. Thus, we get the summand $b_3\widehat{C}_{4, 6, (2, 2, 2, 2), n}^{\vec{b}} = b_3\widehat{C}_{4, 6, n-2}^{\vec{b}} = b_3A_{n-2}$. Next, there is exactly one sub-ballot path from $(2, 1, 1, 0)$ to $(3, 2, 2, 1)$ that does not exceed semisymmetric height $6$: that path is $(\vec{e}_4, \vec{e}_2, \vec{e}_3, \vec{e}_1)$, which has semisymmetric weight $b_3^2$. This observation yields the summand $b_3^2\widehat{C}_{4, 6, (3, 2, 2, 1), n}^{\vec{b}} = b_3^2\widehat{C}_{4, 6, (2, 1, 1, 0), n-1}^{\vec{b}} = b_3^2 B_{n-2}$. In sum, we obtain $B_{n-1} = b_3A_{n-2}+b_3^2B_{n-2}$.

From $
A_{n-1} = b_0b_3 A_{n-2} + b_0 b_3^2 \, B_{n-2}$,
we obtain $B_{n-2} = \frac{A_{n-1} - b_0 b_3 A_{n-2}}{b_0 b_3^2}.$ Substituting this into $B_{n-1} = b_3A_{n-2}+b_3^2B_{n-2}$, we get $B_{n-1} = \frac{1}{b_0} A_{n-1}$. Substituting the last identity into $A_n = b_0b_3A_{n-1}+b_0b_3^2$, we obtain $A_n = (2b_0b_3)A_{n-1}$ for all $n \geq 2$. Together with the initial condition $A_{1} = b_0b_3$, the recursive relation yields $\widehat{C}_{4, 6, 1}^{\vec{b}}  = A_n = b_0b_3(2b_0b_3)^{n-1}$ for $n \geq 1$.
 \end{proof}

A special case of this identity is the following.
\begin{corollary}
\label{cor:2^n}
    For any integer $n\geq 1$, we have
    $\widehat{C}_{4, 6,n}= 2^{n-1}.$
\end{corollary}
\begin{proof}
  We have $\widehat{C}_{4, 6, n}= \widehat{C}_{4, 6, n}^{(1, 1, \dots)}$. Then, we evaluate the expression in Theorem~\ref{thm:42} for $\vec{b} =(1, 1, \dots)$ to obtain the formula.
\end{proof}

For $k=5$, Proposition~\ref{prop:MinHeight} implies that the smallest admissible value is $u=6$. We therefore determine $\widehat{C}_{5,6,n}^{\vec b}$ and $\widehat{C}_{5,7,n}^{\vec b}$. The argument is identical to that of Proposition~\ref{prop:three}.

\begin{proposition}
    For any integer $n \geq 1$, we have $\widehat{C}_{5, 6, n}^{\vec{b}} = \widehat{C}_{5, 7, n}^{\vec{b}} = (b_0b_4)^n$.
\end{proposition}
\begin{proof}
    Observe that the path formed by repeating $(\vec{e}_1, \vec{e}_2, \vec{e}_3, \vec{e}_4, \vec{e}_5)$ $n$ times is the only nonempty 5-dimensional balanced ballot path of length $5n$ with semisymmetric height of at most $7$. We can verify this by applying the same argument as the one used in Proposition~\ref{prop:three} for the sequence $(\vec{e}_1, \vec{e}_2, \vec{e}_3)$ repeated $n$ times. 

    By Proposition~\ref{prop:MinHeight}, the path $P_0$ has semisymmetric height $6$, with $n$ semisymmetric up-steps starting at intermediate points of semisymmetric height $0$ and another $n$ starting at intermediate points of semisymmetric height $4$. This yields $\widehat{C}_{5,6,n}^{\vec b}=\widehat{C}_{5,7,n}^{\vec b}=sswt_{\vec b}(P_0)=(b_0b_4)^n$.
\end{proof}

For $5$-dimensional, $8$-bounded SSWCNs, we prove a similar statement. The proof follows the same
strategy as Theorem~\ref{thm:42}.

\begin{theorem}\label{thm:58}
    For the infinite sequence of integers $\vec{b} = (b_0, b_1, \dots)$ and any integer $n \geq 1$, we obtain
     $$\widehat{C}_{5,8,n}^{\vec{b}} = b_0b_4(2b_0b_4)^{n-1} .$$
\end{theorem}
\begin{proof}

 Let $A_n = \widehat{C}_{5, 8, n}^{\vec{b}}$, and let $B_{n-1}=\widehat{C}_{5, 8, (2,1, 1,1, 0), n}^{\vec{b}}$, the sum of semisymmetric weights $sswt_{\vec{b}}(P)$ over all sub-ballot paths $P$ from $(2, 1,1, 1, 0)$ to $(n, n, n,n,n)$ that do not exceed semisymmetric height $8$. We prove the recursive relations $A_n = b_0 b_4 A_{n-1}+b_0b_4^2B_{n-1}$ and $B_{n-1} = b_4 A_{n-2}+b_4^2 B_{n-2}$ for all $n \geq 2$.

  First, there is only one sub-ballot path from $(0, 0, 0, 0, 0)$ to $(1, 1, 1, 1, 1)$ not exceeding semisymmetric height $8$, and it has weight $b_0b_4$. Then, note that there is only one sub-ballot path from $(0, 0, 0, 0, 0)$ to $(2, 1, 1, 1, 0)$  not exceeding semisymmetric height $8$. That path has semisymmetric weight $b_0b_4^2$. Thus, we obtain $A_n = b_0b_4A_{n-1}+ b_0b_4^2B_{n-1}$.

  Now, note that there is only one sub-ballot path from $(2, 1, 1, 1, 0)$ to $(2, 2, 2, 2, 2)$ with semisymmetric height at most $8$; it contributes summand $b_4A_{n-2}$. Likewise, there is only one sub-ballot path from $(2, 1, 1, 1, 0)$ to $(3, 2, 2, 2, 1)$; it contributes summand $b_4^2B_{n-2}$. In sum, we obtain $B_{n-1} = b_4 A_{n-2}+b_4^2 B_{n-2}$.

From $A_n=b_0b_4A_{n-1}+b_0b_4^2B_{n-1}$, we have $B_{n-2} = \frac{A_{n-1}-b_0b_4A_{n-2}}{b_0b_4^2}$. Substituting into $B_{n-1} = b_4A_{n-2} + b_4^2B_{n-2}$, we get $B_{n-1} =  \frac{A_{n-1}}{b_0}$. Substituting that equality into $A_n=b_0b_4A_{n-1}+b_0b_4^2B_{n-1}$ yields $A_n = (2b_0b_4)\,A_{n-1}$ for $n \geq 2$.  We combine this recursive relation with the initial condition $A_1 = b_0b_4$ to obtain $\widehat{C}_{5, 8, n}^{\vec{b}} = A_n = b_0b_4 (2b_0b_4)^{n-1}$.
\end{proof}

We obtain the following formula for the number of $5$-dimensional balanced ballot paths of length $5n$ and semisymmetric height of at most $8$. 
\begin{corollary}
     For any integer $n \geq 1$, we have
    $\widehat{C}_{5, 8, n}= 2^{n-1}.$
\end{corollary}
\begin{proof}
    We obtain this by first recalling from Section~\ref{sec:ssheight} that $\widehat{C}_{5, 8, n} = \widehat{C}_{5, 8, n}^{(1, 1, \dots)}$ and then evaluating the formula from Theorem~\ref{thm:58} for $\vec{b} = (1, 1, \dots)$.
\end{proof}

Since we see that $\widehat{C}_{4, 6, n} = 2^{n-1}$ and $\widehat{C}_{5, 8, n} = 2^{n-1}$, we suspect that there  are more pairs of integers $k$ and $u$ that yield $\widehat{C}_{k, u, n} = 2^{n-1}$ for $n \geq 1$. Thus, we end the section with a small open problem.
\begin{openproblem}
    Find all integer pairs \( k \geq 2 \) and \( u \geq \left\lceil\frac{k}{2}\right\rceil \cdot \left\lfloor\frac{k}{2}\right\rfloor \) such that \( \widehat{C}_{k, u, n} = 2^{n-1} \).
\end{openproblem}
\section{The \texorpdfstring{$k$}{k}-dimensional Semisymmetric Height Triangle}\label{sec:heighttriangles}
Since we introduced the $k$-dimensional $u$-bounded semisymmetric Catalan number $\widehat{C}_{k, u, n}$ (Definition~\ref{defn:k-dims-boundCat}), the refinement of balanced ballot paths, it is natural to refine the enumeration of $k$-dimensional balanced ballot paths, namely the $k$-dimensional Catalan numbers, by this parameter. For Dyck paths, which correspond to the case $k=2$, the resulting sequence is A080936 in the OEIS \cite{oeis}. In our previous paper \cite{MR5034574}, we counted $k$-dimensional balanced ballot paths $P$ by their height $h_k(P)$; this motivates the semisymmetric height triangle, which refines these paths by semisymmetric height.

\begin{definition} We define $D_{k,u,n}^{\,\prime}$ to be the number of $k$-dimensional ballot paths of length $kn$ such that their maximal semisymmetric height is exactly $u$, i.e. for at least one intermediate point $g_k(\vec{x})=u$, but for no points $g_k(\vec{x}) > u$. For fixed $k$, we denote the table of values of $D_{k, u, n}^{\,\prime}$ over $u$ and $n$. The table of values of $D_{k,u , n}^{\,\prime}$ is the \textbf{$k$-dimensional  semisymmetric height triangle}. \end{definition}
In particular, for the $2$-dimensional case, the triangle coincides with the sequence counting Dyck paths by exact height (A080936 of the OEIS \cite{oeis}). Indeed, consider the bijective map between $2$-dimensional balanced ballot paths and Dyck paths (from Section~\ref{sec:bbp}): given balanced ballot path $P$, replace each instance of step $\vec{e}_1$ with step $(1, 1)$ and each instance of step $\vec{e}_2$ with step $(1, -1)$. The semisymmetric height of the $i$th intermediate point of $P$ is equal to the height of the $i$th intermediate point of the corresponding Dyck path. Hence, a $2$-dimensional balanced ballot path $P$ has semisymmetric height $u$ if and only if its corresponding Dyck path $P'$ has height (maximum $y$-coordinate attained by intermediate point) $u$.

\begin{remark}
We use the $D_{k, u, n}'$ to avoid confusion with $D_{k, u, n}$, the latter being the notation used for the enumeration of $k$-dimensional balanced ballot paths by their exact height $h_k$ from \cite{MR5034574}.\end{remark}

From the definition and results in Section~\ref{sec:BasicObservations}, we make several observations on the nonzero entries of $D_{k, u, n}'$. By the above definition, we obtain the equality $$D_{k, u, n}' = \widehat{C}_{k, u, n} - \widehat{C}_{k, u-1, n}.$$

Note that entry $D_{k,\left\lfloor  \frac{k}{2} \right\rfloor \left\lceil  \frac{k}{2} \right\rceil n, n}'$ is the rightmost nonzero entry of the $n$th row of the $k$-dimensional semisymmetric height triangle. Indeed, Proposition~\ref{prop:maxheight} implies that $\left\lfloor  \frac{k}{2} \right\rfloor \left\lceil  \frac{k}{2} \right\rceil n$ is the largest possible semisymmetric height of a $k$-dimensional balanced ballot path of $kn$ steps. Furthermore, entry $D_{k,\left\lfloor  \frac{k}{2} \right\rfloor \left\lceil  \frac{k}{2} \right\rceil, n}^{\, \prime}$ is the leftmost entry of the $n$th row of the $k$-dimensional semisymmetric height triangle. This is because Proposition~\ref{prop:MinHeight} tells us that $\left\lfloor  \frac{k}{2} \right\rfloor \left\lceil  \frac{k}{2} \right\rceil$ is the smallest possible semisymmetric height of a $k$-dimensional balanced ballot path of $kn$ steps

We first examine the 3-dimensional semisymmetric height triangle, which is the new sequence A393573 of the OEIS \cite{oeis}. Using the Python script in the OEIS sequence, we compute the first 5 nonzero rows of the 3-dimensional semisymmetric height triangle for $k=3$; these values of $D_{3, u, n}^{\,\prime}$ are illustrated in Table~\ref{table:3dtriangle2}. From this table, we observe that the 3-dimensional semisymmetric height triangle is distinct from the enumeration of $3$-dimensional balanced ballot paths by height from our previous paper \cite{MR5034574} (A387912 in the OEIS \cite{oeis}).

\begin{table}[H]
    \centering
$\begin{array}{c|ccccccccccc} n \backslash u & 2 & 3 & 4 & 5 & 6 & 7 & 8 & 9 & 10 & 11 & 12\\
\hline
1 & 1 \\
2 & 1 & 0 & 4 \\ 3 & 1 & 0 & 20 & 0 & 21
\\ 4 & 1 & 0 & 88 & 0 & 252 & 0 & 121
\\ 5 & 1 & 0 & 376 & 0 & 2354 & 0 & 2547 & 0 & 728
\end{array}$

    \caption{Some entries $D_{3, u, n}'$ of the $3$-dimensional semisymmetric height triangle. The variable $u$ stands for semisymmetric height.}
    \label{table:3dtriangle2}
\end{table}

In Table~\ref{table:3dtriangle2} for the $3$-dimensional semisymmetric triangle, the columns corresponding to odd values of $u$ are zero. For odd $u$ and $k$, we have $D_{k, u, n}^{\,\prime} = 0$. Indeed, we find that for odd $k$, the semisymmetric height function $g_k(\vec{x}) = \sum_{i=1}^{k}(k+1-2i)x_i$ has only even coefficients. This means, for odd $k$, any intermediate point of a $k$-dimensional balanced ballot path has even weight. Thus, if $k$ is odd, any $k$-dimensional balanced ballot path has an even semisymmetric height

Similarly, we consider the $4$-dimensional semisymmetric height triangle as the new sequence A393594 in the OEIS \cite{oeis}. Using the Python script in the OEIS entry, we compute the first 4 nonzero rows of the $4$-dimensional semisymmetric height triangle, as illustrated in Table~\ref{table:4dtriangle2}. From the table, we see that the $4$-dimensional semisymmetric height triangle is different from the $4$-dimensional balanced-Ballot-Path-Height triangle from Section 5 of \cite{MR5034574} (A387987 in the OEIS \cite{oeis}).
\begin{table}[H]
    \centering
  \[
\begin{array}{c|cccccccccccccc} n \backslash u & 4 & 5 & 6 & 7 & 8 & 9 & 10 & 11 & 12 & 13 & 14 & 15 & 16\\
\hline
1 & 1\\
2 & 1 & 0 & 1 & 8 & 4 \\
3 & 1 & 0 & 3 & 69 & 48 & 30 & 151 & 135 & 25 \\4 & 
1 & 0 & 7 & 533 & 553 & 583 & 4299 & 5051 & 1930 & 4288 & 4819 & 1764 & 196
\end{array}
\]

    \caption{Some entries $D'_{4, u, n}$ of the 4-dimensional semisymmetric height triangle. Parameter $u$ stands for semisymmetric height.}
    \label{table:4dtriangle2}
\end{table}

In Table~\ref{table:3dtriangle2}, the last nonzero entry of each row is bounded from below by the $n$th Catalan number $C_n$. Furthermore, in Table~\ref{table:4dtriangle2}, the last nonzero entries of each row form the sequence of squares of Catalan numbers (A001246 in the OEIS \cite{oeis}). From this, one may conjecture that for general $k$, we can bound the last nonzero entry of each row of the $ k$-dimensional semisymmetric height triangle lower bound using an expression in terms of the multidimensional Catalan numbers (Definition~\ref{def:multidim}) from earlier. We confirm this intuition below.
\begin{proposition}\label{prop:Bound}
    For each $k \geq 2$ and $n \geq 1$, we have \[D'_{k, \lceil \frac{k}{2}\rceil \lfloor \frac{k}{2}\rfloor n, n} \geq {C}_{\lceil \frac{k}{2}\rceil, n}{C}_{\lfloor \frac{k}{2}\rfloor, n}.\]
\end{proposition}
\begin{proof}
    We count the number of $k$-dimensional balanced ballot paths consisting first of $\left\lceil k/2 \right\rceil n$ semisymmetric up-steps and then consisting of $\left\lfloor \frac{k}{2} \right\rfloor n$ semisymmetric down-steps. We use $\mathcal{S}$ to denote this class of $k$-dimensional balanced ballot paths. All paths in this class $\mathcal{S}$ have semisymmetric height $\sum_{j=1}^{\left\lceil \frac{k}{2} \right\rceil}(k-2j+1)n = \left\lceil \frac{k}{2}\right\rceil \left\lfloor \frac{k}{2} \right\rfloor n$.  Thus, we observe $|\mathcal{S}| \leq D'_{k, \left\lceil \frac{k}{2}\right\rceil \left\lfloor \frac{k}{2}\right\rfloor n, n} $.

    We now count the number of paths in $\mathcal{S}$. First, there are ${C}_{\left\lceil \frac{k}{2}\right\rceil, n}$ paths consisting of only the $\left\lceil k/2 \right\rceil n$ semisymmetric up and neutral steps such that for any intermediate point $x_1 \geq x_2 \geq \dots \geq x_{\left\lceil k/2 \right\rceil}$. Then, there are $C_{\left\lfloor \frac{k}{2}\right\rfloor, n}$ paths of the $\left\lfloor k/2 \right\rfloor n$ consisting of down-steps such that for any intermediate point $x_{\left\lceil \frac{k}{2}\right\rceil+1} \geq x_{\left\lceil \frac{k}{2}\right\rceil+2}  \geq \dots \geq x_{k}$. Thus, multiplying these two multidimensional Catalan numbers yields ${C}_{\left\lceil \frac{k}{2}\right\rceil, n}{C}_{\left\lfloor \frac{k}{2}\right\rfloor, n}  = |\mathcal{S}| \leq D'_{k, \left\lceil \frac{k}{2}\right\rceil \left\lfloor \frac{k}{2}\right\rfloor n, n}$.
\end{proof}

For even $k$, this lower bound on $D'_{k, \left\lceil \frac{k}{2}\right\rceil \left\lfloor \frac{k}{2}\right\rfloor n, n}$ is tight.
\begin{proposition}\label{prop:DprimeCk2}
    For any positive integer $n$ and even $k \geq 2$, we have $$D'_{k, \left(\frac{k^2}{4}\right) n, n} = ({C}_{k/2, n})^2.$$
\end{proposition}
\begin{proof}
    We first claim that a $k$-dimensional balanced ballot path  $P$ reaches semisymmetric height $\left(\frac{k}{2}\right)^2 n$ if and only if the first $\frac{k}{2} \cdot n$ steps are semisymmetric up-steps of (i.e., are steps of form $\vec{e}_i \in \{ \vec{e}_1, \vec{e}_2, \ldots, \vec{e}_{\frac{k}{2}} \}$) and the next $\frac{k}{2}n$ are semisymmetric down-steps (i.e., are steps of form $\vec{e}_i \in \{\vec{e}_{k/2+1}, \vec{e}_{k/2+2}, \dots, \vec{e}_k$) of $\vec{e}_j\}$).
    
    If the first $\frac{k}{2} \cdot n$ steps of $P$ are semisymmetric up-steps, the maximum height attained by an intermediate point of $P$ is $\sum_{i=1}^{\frac{k}{2}}(k+1-2i)n = \frac{k^2}{4}n$, which is the height of the $\frac{k}{2}n$th intermediate point. Hence, if the first $\frac{k}{2} \cdot n$ steps of $P$ are semisymmetric up-steps, the semisymmetric height of $P$ is  $\frac{k^2}{4}n$. 
    
    On the other hand, suppose there exists a semisymmetric down-step $\vec{s}_{i_{down}}$ within the first $\frac{k}{2}n$ steps. Let  $i_{max} \in \{0, 1, 2, \dots, kn\}$ denote the index such that the semisymmetric height $i_{max}$th intermediate point is the semisymmetric height of $P$, i.e., is the highest possible one attained by an intermediate point of $P$. Let $\vec{x} = (x_1, x_2, \dots, x_k)$ denote the $i_{max}$th intermediate point. In this setting, either the $i_{max}$th step is before the last semisymmetric up-step, or it is the last semisymmetric up-step. (If the $\vec{s}_{i_{max}}$th step is after the last semisymmetric up-step, then $\vec{s}_{i_{max}}$ is a down-step and hence the $(i_{max}-1)$th intermediate point has a higher semisymmetric height, contradicting the definition of $i_{max}$.) If the $i_{max}$th step is before the last semisymmetric up-step, we find that $g_k(\vec{v}_{i_{max}}) = \sum_{i=1}^{k}(k+1-2i)x_i \leq \sum_{i=1}^{\frac{k}{2}}(k+1-2i)x_i < \sum_{i=1}^{\frac{k}{2}}(k+1-2i) n = \frac{k^2n}{4}$ with the strict inequality following from $i_{max}$ being before the last semisymmetric up-step.  Now, suppose that the $i_{max}$th step is the last semisymmetric up-step. Then we find 
    $g_k(\vec{v}_{i_{max}})= \sum_{i=1}^{k}(k+1-2i)x_i < \sum_{i=1}^{\frac{k}{2}}(k+1-2i)x_i = \sum_{i=1}^{\frac{k}{2}}(k+1-2i) n = \frac{k^2n}{4}$ with the first strict inequality being due to how there is a semisymmetric down-step before the last semisymmetric up-step. Therefore the semisymmetric height of $P$, $g_k(P)$, is less than $\frac{k^2n}{4}$.

    Hence, a $k$-dimensional balanced ballot path $P$ reaches semisymmetric height $\frac{k^2}{4}n$ if and only if the first $\frac{k}{2}n$ steps are semisymmetric up-steps and the rest are down-steps. This fact implies $D'_{k, \frac{k^2}{4}n, n}$ is the number of such paths.

    We now compute $D'_{k, \frac{k^2}{4}n, n}$. The number of possible length $\frac{k}{2}n$ prefixes of such $k$-dimensional balanced ballot paths is the $k/2$-dimensional Catalan number $C_{k/2, n}$, which is the number of $\frac{k}{2}$-dimensional of $\frac{kn}{2}$ steps, because of the ballot property (see Definition~\ref{defn:k-dim-path}) of balanced ballot paths. For a similar argument on the ballot property, the number of length $\frac{k}{2}n$ suffixes of such $k$-dimensional balanced ballot paths is the $k/2$-dimensional Catalan number $C_{k/2, n}$. Combining these two, we conclude with the equality $D'_{k,\lceil \frac{k}{2} \rceil \lfloor \frac{k}{2} \rfloor,n} = (C_{k/2,n})^2$.
\end{proof}
\begin{example}
    As illustrated in Table~\ref{table:4dtriangle2}, Proposition~\ref{prop:DprimeCk2} implies identity $D'_{4, 4n, n} = \widehat{C}_{2, n}^2 = C_n^2$.
\end{example}

On the other hand, the bound on $D'_{k, \left\lceil k/2\right\rceil \left\lfloor k/2 \right\rfloor n, n}$ from Proposition~\ref{prop:Bound} does not seem to be tight for odd $k$. We verify this for $k=3$ by proving the exact formula for $D'_{3, 2n, n}$. In the Proposition below, we find that $ D'_{3, 2n, n}$ is equal to the $n$th term of the sequence A274969 of the OEIS \cite{oeis}. 

\begin{proposition}\label{prop:NotTightExample}
    For each $n \geq 1$, we have $D'_{3, 2n, n} = \binom{3n}{n} - 2\binom{3n}{n-1} + \binom{3n}{n-2}$.
\end{proposition}
\begin{proof}
     Observe that the expression $\binom{3n}{n} - 2\binom{3n}{n-1} + \binom{3n}{n-2}$ is the number of paths from $(0, 0)$ to $(0, 0)$ within $\mathbb{Z}_{\geq 0}^2$ consisting of $3n$ steps from $\{A=(1, 0), B= (-1, 1), C=(0, -1)\}$. Call the set of such paths $Q_{n}$. Thus, it suffices to show that the set of $3$-dimensional balanced ballot paths counted by $D'_{3, 2n, n}$, i.e., the $3n$ steps with exact semisymmetric height  $2n$, is equinumerous to $Q_n$. We prove this fact via a bijective proof.

    Consider the following map $f$ from the set  of $3$-dimensional balanced ballot paths counted by $D'_{3, 2n, n}$ to the set $Q_n$: for balanced ballot path $P$ in the domain, we construct path $q \in Q_n$ by letting the $i$th step of $q$ be $A$ if the $i$th step of $P$ is $\vec{e}_1$, $B$ if the $i$th step of $P$ is $\vec{e}_2$, and $C$ if the $i$th step of $P$ is $\vec{e}_3$. 
    
    We confirm that, for any 3-dimensional balanced ballot path $P$ of length $3n$ with semisymmetric height $2n$, the path  $f(P)$ is in $Q_n$. First, the intermediate points of $Q_n$ are all in $\mathbb{Z}_{\geq 0}^2$. Indeed, ballot property of $P$ guarantees that, in any prefix of the sequence $f(P)$, the number of steps equal to $A = (1, 0)$ is at least that of steps equal to $B = (-1, 1)$, and that the number of steps equal to $B = (-1, 1)$ is at least that of steps equal to $C = (0, 1)$.

    Thus, to show $f(P)$ is in $Q_n$, it suffices to show that the last step equal to $A$ in path $f(P)$ is before the first step equal to $C$. We prove an equivalent statement: the first $\vec{e}_3$ step (i.e., the first semisymmetric down-step) in $P$ must be after the last $\vec{e}_1$ step (i.e., last semisymmetric down-step). Suppose, for the sake of contradiction, that this statement is not true. Let $i_{max} \in \{0, 1, \dots, kn\}$ be the smallest possible integer so that the semisymmetric height of $P$, i.e., the maximum possible semisymmetric height attainable by an intermediate point of $P$, is that of the $i_{max}$th intermediate point of $P$. Let $\vec{v}_{i_{max}}=(x_1, x_2, x_3)$ be the $i_{max}$th intermediate point of $P$. We know that $i_{max}$th step must not be a semisymmetric down-step; otherwise, the $(i_{max}-1)$th intermediate point of $P$ must have a higher semisymmetric height than the $i_{max}$th. We also know the $i_{max}$th step of $P$ must be an up-step; otherwise the $(i_{max}-1)$th intermediate point would have a semisymmetric height equal to that of $P$, contradicting the definition of $i_{max}$. We then obtain that $g_3(\vec{v}_{i_{max}})=2x_1-2x_3 < 2n$. If the $i_{max}$th step is not the last semisymmetric up-step, then we find that $g_3(\vec{v}_{i_{max}})=2x_1-2x_3 \leq 2x_1 < 2n$. Otherwise, if the $i_{max}$th step is the last semisymmetric up-step, we see $g_3(\vec{v}_{i_{max}})=2x_1-2x_3 < 2x_1 = 2n$, because of the semisymmetric down-step that came before the last semisymmetric up-step. Therefore, $g_k(\vec{v}_{i_{max}}) < 2n$ is the semisymmetric height of $P$, and a contradiction is reached. 

    In sum, we have that $f(P)$ is in $\mathbb{Z}_{\geq 0}^2$, and the last step equal to $A$ in $f(P)$ is before the first equal to $C$.  Hence, the path $f(P)$ is in $Q_n$.
    
    We now show that the function $f$ is a bijection because it has an inverse function $g$. Consider the function $g$, defined as follows: for $q \in Q_n$, if the $i$th step of $q$ is $A$ then the $i$th step of $g(q)$ is $\vec{e}_1$, if the $i$th step of $q$ is $B$ then the $i$th step of $g(q)$ is $\vec{e}_2$, and if the $i$th step of $q$ is $C$ then the $i$th step of $g(q)$ is $\vec{e}_3$. 
    
    We observe that the range of $g$ is a subset of $Q_n$. Firstly, there are $n$ steps in $g(q)$ equal to $\vec{e}_1, \vec{e}_2$, and $\vec{e}_3$ since there are $n$ steps equal to $A$, $B$, and $C$ in $q$. Furthermore, path $g(q)$ exhibits the ballot property. Otherwise, for some prefix of $q$, there are 1) more instances of $B$  than $A$ or 2) more of $C$ than $A$, and that would imply that $q$ is not in $\mathbb{Z}_{\geq 0}^2$, a contradiction of how $q \in Q_n$.
    
    Furthermore, the balanced ballot path $g(q)$ has a semisymmetric height of exactly $2n$. Indeed, the last step equal to $A$ in path $q$ is before the first step equal to $C$, which means the last instance of $\vec{e}_1$ in $g(q)$ is before the first of $\vec{e}_3$. Consequently, the semisymmetric height of the intermediate point resulting from the last $\vec{e}_1$ step is $2n$.

    Finally, function $g$ is the inverse of $f$. Indeed, for any $q \in Q_n$ and  3-dimensional balanced ballot path $P$ of $3n$ steps with semisymmetric height $2n$, we obtain $f(g(q)) = q$ and $g(f(P))= P$ because $f$ and $g$ relabel steps.

    Because $f$ is a bijection from the set of $3$-dimensional balanced ballot paths of length $3n$ with semisymmetric height $2n$ to $Q_n$, we find that $D_{3, 2n, n}' = |Q_n| = \binom{3n}{n} - 2\binom{3n}{n-1}+ \binom{3n}{n-2}$.
\end{proof}

Now, observe that Proposition~\ref{prop:NotTightExample} implies the inequality $$D_{3, 2n, n}' = \binom{3n}{n} - 2\binom{3n}{n-1}  + \binom{3n}{n-2} > C_{ \left\lceil \frac{3}{2} \right\rceil, n} \cdot C_{ \left\lfloor \frac{3}{2} \right\rfloor, n} = C_{2, n}C_{1, n} = C_n.$$  Thus, we conclude the bound from Proposition~\ref{prop:Bound} for $D'_{k, \left\lceil \frac{k}{2} \right\rceil \left\lfloor \frac{k}{2} \right\rfloor n, n}$ is not tight for $k=3$.

\section{The \texorpdfstring{$k$}{k}-dimensional Semisymmetric Narayana Number}\label{sec:Narayana}
Since balanced ballot paths admit semisymmetric up-steps and down-steps (Definition~\ref{defn:up-step-down-step}), it is natural to refine their enumeration accordingly. The classical Narayana number $N(n,\alpha)$ (A001263 in the OEIS \cite{oeis}) counts Dyck paths of length $2n$ with exactly $\alpha$ peaks. Sulanke’s analog $N(k,\alpha,n)$ \cite{multidimnarayana} counts $k$-dimensional balanced ballot paths of length $kn$ with exactly $\alpha$ occurrences of a step $\vec{e}_i$ immediately followed by a step $\vec{e} _ {j} $ for some $i>j$. In contrast, \cite{MR5034574} we defined $N'{k,\alpha,n}$ to count those with exactly $\alpha$ steps of the form $\vec{e}_1$ immediately followed by a step other than $\vec{e}_1$. This motivates the introduction of the $k$-dimensional semisymmetric Narayana numbers, defined in terms of peaks arising from semisymmetric up- and down-steps.

After introducing the $k$-dimensional semisymmetric Narayana numbers, we compute numerical examples in the $3$- and $4$-dimensional cases. These examples show that the semisymmetric Narayana numbers differ from the earlier analogs in \cite{multidimnarayana, MR5034574}. We also prove that the $k$-dimensional semisymmetric Narayana numbers and the $k$-dimensional semisymmetric height triangle from Section~\ref{sec:heighttriangles} share a subsequence of terms.

To construct our new $k$-dimensional analog of the Narayana numbers, we begin by providing an alternative definition of peaks.
\begin{definition}
Let $P$ be a $k$-dimensional balanced ballot path. In this paper, a \textbf{semisymmetric peak} in $P$ is any instance of a semisymmetric up-step, i.e. any step in $\{\vec{e}_1, \vec{e}_2, \dots, \vec{e}_{\left\lfloor \frac{k}{2} \right\rfloor}\}$, followed immediately by a semisymmetric down-step, i.e. any step in $\{\vec{e}_{\left\lceil \frac{k
}{2} \right\rceil+1}, \vec{e}_{\left\lceil \frac{k
}{2} \right\rceil+2}, \dots, \vec{e}_{k}\}$.
\end{definition}

We use the term semisymmetric peak because our notions of up-steps and down-steps originate from our notion of semisymmetric height (see Section~\ref{sec:bbp}). We also choose this term to avoid confusion with the $k$-dimensional peaks from our previous paper \cite{MR5034574}. There, a peak is a step equal to $\vec{e}_1$ immediately followed by one that is not $\vec{e}_1$.

\begin{example}
Consider the $4$-dimensional balanced ballot path $(\vec{e}_1, \vec{e}_1, \vec{e}_2, \vec{e}_3, \vec{e}_4, \vec{e}_2, \vec{e}_3, \vec{e}_4)$. This path has exactly two semisymmetric peaks: the two instances of $\vec{e}_2$ that are immediately followed by a step equal to $\vec{e}_3$.
\end{example}

Building on our new notion of peaks, we now introduce the $k$-dimensional semisymmetric Narayana number, an analog of the Narayana numbers.

\begin{definition}
Denote by $N_{k,\alpha,n}''$ the number of $k$-dimensional balanced ballot paths with $kn$ steps and $\alpha$ semisymmetric peaks. We call $N_{k, \alpha, n}''$ the \textbf{$k$-dimensional semisymmetric Narayana number}. We call the table of nonzero values of $N_{k, \alpha, n}''$ the \textbf{$k$-dimensional semisymmetric Narayana triangle}.
\end{definition}
\begin{remark}
We use the notation $N_{k, \alpha, n}''$ to avoid confusion with the $k$-dimensional Narayana numbers from \cite{MR5034574}, which were denoted by $N_{k, \alpha ,n}'$.\end{remark}

In particular, we obtain $N_{2,\alpha,n}'' = N(n, \alpha)$. This result follows from the bijection between $2$-dimensional balanced ballot paths and Dyck paths (see Section~\ref{sec:bbp}). The $i$th step of a $2$-dimensional balanced ballot path $P$ is a semisymmetric peak if and only if it equals $\vec{e}_1$ and is immediately followed by $\vec{e}_2$. This is equivalent to the $i$th step of the corresponding Dyck path being a peak, which means a step $(1, 1)$ immediately followed by a step $(1, -1)$. Therefore, the number of Dyck paths of length $2n$ with $\alpha$ peaks equals the number of $2$-dimensional balanced ballot paths of length $2n$ with $\alpha$ semisymmetric peaks.

Consider the 3-dimensional semisymmetric Narayana triangle, which is the new sequence A393747 in the OEIS \cite{oeis}. Using the Python script in the OEIS entry, we compute the first 5 nonzero rows of the 3-dimensional semisymmetric Narayana triangle, shown in Table~\ref{tab:3dnarayana2}. The table shows that the 3-dimensional semisymmetric Narayana numbers $N_{k, \alpha,n}''$ differ from both the values of $N_{k, \alpha, n}'$ presented in \cite{MR5034574}, and $N(3, \alpha, n)$ reported by Sulanke \cite{multidimnarayana} (A087647 from the OEIS \cite{oeis}).

\begin{table}[H]
\[
\begin{array}{c|cccccc}
n \backslash \alpha & 0 & 1 & 2 & 3 & 4  & 5\\
\hline
0 & 1 \\
1 & 1 \\
2& 4 & 1 & \\
3& 25& 16 & 1\\
4& 196 &  221 & 49 & 1\\
5& 1764 & 2976 & 1161 & 104 & 1\\
6& 17424 & 40105 & 24972 & 4786 & 228 & 1\\
\end{array}
\]
\captionof{table}{Values of $N_{3, \alpha ,n}''$, the $3$-dimensional semisymmetric Narayana triangle. Parameter $\alpha$ stands for the number of semisymmetric peaks.}
\label{tab:3dnarayana2}
\end{table}

\medskip
Similarly, consider the 4-dimensional semisymmetric Narayana triangle, which is the new sequence A393571 in the OEIS \cite{oeis}. Using the Python script in the OEIS entry, we compute the first 5 nonzero rows of the 4-dimensional semisymmetric Narayana triangle. These appear in Table~\ref{tab:4dnarayana2}. The table shows that the $4$-dimensional semisymmetric Narayana triangle differs from the 4-dimensional Narayana triangle in \cite{MR5034574} (A387936 in the OEIS  \cite{oeis}).
\begin{table}[H]
\[
\begin{array}{c|cccccccccc}
n \backslash \alpha & 0 & 1 & 2 & 3 & 4 & 5 & 6 & 7 & 8 & 9 \\
\hline
0 & 1\\
1 & 0 & 1 & \\
2& 0 & 4 & 9 & 1\\
3 & 0 & 25 & 175 & 256 & 45 & 1\\
4& 0 & 196 & 2828 & 9285 & 9038 & 2514 & 162 & 1\\
5 & 0 & 1764 & 43508 & 274138 & 613545 & 533694 & 176091 & 19541 & 522 & 1\\
\end{array}
\]
\captionof{table}{Values of $N_{4, \alpha ,n}''$, the $4$-dimensional semisymmetric Narayana triangle. Parameter $\alpha$ stands for the number of semisymmetric peaks.}
\label{tab:4dnarayana2}
\end{table}

For even values of $k$, we have $N_{k, n, 0}'' = 0$ for all $n \geq 1$. This identity holds because all nonempty $k$-dimensional balanced ballot paths must contain only semisymmetric up-steps and down-steps (see Definition~\ref{defn:up-step-down-step}).

Additionally, we show below that the subsequence $N_{k, 1, n}''$ of the $k$-dimensional semisymmetric Narayana triangle corresponds to the sequence $D_{k, \frac{k^2n}{4}, n}'$. This sequence represents the rightmost entries in the $k$-dimensional semisymmetric height triangle.

\begin{theorem}\label{thm:RelatingTheTriangles}
For even $k$, we have $$N_{k,1,n}'' = ({C}_{\left(\frac{k}{2}\right),n})^2 = D_{k, \left(\frac{k}{2}\right)^2 n, n}'.$$
\end{theorem}
\begin{proof}
We first prove that $N_{k,1,n}'' = ({C}_{\left(\frac{k}{2}\right),n})^2$. If the path $P$ has one semisymmetric peak, then there is only one pair of consecutive steps $\{\vec{e}_i,\vec{e}_j\}$ where $i \in \{1, 2, \dots, \frac{k}{2}\}$ and $j \in \{\frac{k}{2}+1,\frac{k}{2}+2, \dots, k\}$. This occurs if and only if both the first $\frac{k}{2} \cdot n$ steps are semisymmetric up-steps of $\vec{e}i$ (with $\vec{e}_i \in \{ \vec{e}_1, \vec{e}_2, \ldots, \vec{e}_{k/2} \}$) and the next $\frac{k}{2}n$ are semisymmetric down-steps of $\vec{e}j$. There are no semisymmetric neutral steps since $k$ is even. In every such balanced ballot path $P$, the semisymmetric up-steps admit $C_{k/2,n}$ distinct orderings, and likewise the semisymmetric down-steps admit $C_{k/2,n}$ distinct orderings. Therefore, there are $N_{k,1,n}'' = (C_{k/2,n})^2$ $k$-dimensional balanced ballot paths of length $kn$ that have only one semisymmetric peak.

The equality $({C}_{\left(\frac{k}{2}\right),n})^2 = D_{k, \left(\frac{k}{2}\right)^2 n, n}'$ follows from Proposition~\ref{prop:DprimeCk2}, completing our proof.
\end{proof}
\section{A Future Direction via Standard Young Tableaux}\label{sec:FurtherDir}
In this paper, using a new notion of height of balanced ballot paths, we defined and studied $k$-dimensional SSWCNs, a new weighted analog of the multidimensional Catalan numbers, as well as the closely related $k$-dimensional $u$-bounded SSWCNs. However, this study of multidimensional analogs of weighted Catalan numbers was motivated and defined in terms of balanced ballot paths rather than by standard Young tableaux. The latter are another important class of objects counted by the multidimensional Catalan numbers and are an extensively studied topic in their own right (see Chapter 7 of Stanley \cite{MR4621625} for a thorough exposition). Thus, we propose the following open problem.

\begin{openproblem}\label{op:probl1}
Provide a natural $k$-dimensional generalization of the weighted Catalan numbers that is in terms of standard Young tableaux and is equivalent to neither the $k$-dimensional weighted Catalan numbers from \cite{MR5034574} nor the  $k$-dimensional SSWCNs and their $u$-bounded variants (Definitions~\ref{defn:kdimweighted} and~\ref{defn:k-dims-boundCat}).
\end{openproblem}

\subsection{Standard Young Tableaux (SYTs)}\label{sec:introsyt}
\subsubsection{Definitions}
Before we discuss an approach to the open problem, we first formally introduce the standard Young tableau. A \textbf{partition} $\lambda$ of $n$, often denoted by $\lambda \vdash n $, is a nonincreasing sequence of nonnegative integers whose sum is $n$. A \textbf{Young diagram} of $\lambda$ is a graphical representation of the partition where the $i$th row contains $\lambda_i$ boxes. A \textbf{standard Young tableau (SYT)} $T$ of shape $\lambda \vdash n$ is a Young diagram of $\lambda$ where each box contains an integer from $1,2, \dots, n$, each appearing once, such that both rows and columns are strictly increasing. An \textbf{SYT of shape $k \times n$} is an SYT of shape $(n, n, \dots,n)$, the sequence with $n$ repeated $k$ times.
\begin{example}\label{ex:SYTeg}
Figure~\ref{fig:SYTeg} illustrates a standard Young tableau $T_{ex}$ of shape $3 \times 4$ (i.e., shape $(4, 4, 4)$). The boxes containing the bolded entries form a subtableau $T_{ex}'$ with $4$ boxes of $T_{ex}$. We will reference this tableau and its subtableau as our running example for the section. \begin{figure}[H]
\centering
\begin{ytableau}
\textbf{1} & \textbf{2} & \textbf{4} & 7 \\
\textbf{3} & 5  & 6 & 8\\
9 & 10  & 11 & 12\\
\end{ytableau}
\caption{A Standard Young Tableau of shape $3 \times 4$.}
\label{fig:SYTeg}
\end{figure}
\end{example}

Given an SYT $T$ of shape $k \times n$, a \textbf{subtableau $T'$ with $n'$ boxes} of $T$ is the standard Young tableau consisting of the boxes of $T$ containing numbers $1,2, \dots n'$.

SYTs are a natural combinatorial setting for defining multidimensional analogs of weighted Catalan numbers. We contend this by observing that the balanced ballot paths and their intermediate points have analogs in terms of SYTs.

Specifically, consider the following known bijection $\beta$ from Stanley’s textbook \cite{MR4621625} from the $k$-dimensional balanced ballot paths of length $kn$ and standard Young tableaux of shape $k \times n$: for each $i$ in $\{1, 2,\dots, kn\}$, if the $i$th step of balanced ballot path $P$ is $\vec{e}_j$, then fill in the right most box in row $j$ with the number $i$.
\begin{example}
Consider the balanced ballot path $P = (\vec{e}_1, \vec{e}_1, \vec{e}_2, \vec{e}_1, \vec{e}_2, \vec{e}_2, \vec{e}_1, \vec{e}_2, \vec{e}_3, \vec{e}_3, \vec{e}_3, \vec{e}_3)$. The corresponding standard Young tableau $\beta(P)$ of shape $k\times n$, is $T$ from Example~\ref{ex:SYTeg}.
\end{example}
Furthermore, there is a bijection $\beta_P'$ between the intermediate points of a balanced ballot path $P$ and the subtableaux of $\beta(P)$: the $i$th intermediate point of $P$ maps to the subtableau with $i$ boxes of $\beta(P)$.
\begin{example} The $4$th intermediate point of $P= (\vec{e}_1, \vec{e}_1, \vec{e}_2, \vec{e}_1, \vec{e}_2, \vec{e}_2, \vec{e}_1, \vec{e}_2, \vec{e}_3, \vec{e}_3, \vec{e}_3, \vec{e}_3)$, which is $\vec{v}_4 = 3\vec{e}_1+3\vec{e}2$, corresponds to the subtableau $T_{ex}'$ from Example~\ref{ex:SYTeg} (the bolded subregion of Figure~\ref{fig:SYTeg}). \end{example}
\subsection{An Approach using Ascents and Descents}\label{sec:PotentialDirection}
We now propose a potential approach for solving Open Problem~\ref{op:probl1}: use ascents and descents of SYTs to define a new statistic and employ the statistic to define weights of SYTs of shape $k \times n$ and generalizations of weighted Catalan numbers.

Ascents, descents, and their variants have been widely studied in the context of SYTs (see, for instance, \cite{MR4190402, MR4867903, MR4621625}). For SYT $T$ of shape $\lambda \vdash n$ and integer  $i \leq n-1$, integer $i$ is a \textbf{descent} of the standard Young tableau $T$ if $i+1$ appears in a lower row than $i$ does. We say that $i$ is an \textbf{ascent} if $i+1$ appears in the same or higher row.

\begin{example}
In the SYT $T_{ex}$ from Example~\ref{ex:SYTeg}, the integer $2$ is a descent because $3$ is in the row below $2$. The integer $3$ is an ascent because $4$ is in the row above.
\end{example}

Using the notion of ascents and descents, we propose a new statistic, which we call the tally, on SYTs. For SYT $T$, we define the \textbf{tally $\tau(T)$ of $T$} to be the number of ascents in $T$ minus the number of descents. \begin{example}\label{ex:SYTtally}
In the SYT $T_{ex}$ from Example~\ref{ex:SYTeg}, the tally of $T_{ex}$ is $\tau(T_{ex}) = 8-4=4$. In $T_{ex}$, the ascents are $1, 3, 5, 6, 9, 10, 11,$ and 12, and the descents are integers $2,4, 7 $ and $8$.
\end{example}

Even though the subtableaux of $\beta(P)$ correspond to intermediate points of $P$ under the bijection $\beta$, the tally of the full tableau does not necessarily vanish (Example~\ref{ex:SYTtally}), unlike the semisymmetric height of the endpoint $(n, n, \dots, n)$ of $P$. This observation suggests any weight function built from the tally will differ from the semisymmetric weight (Definition~\ref{defn:ssweight}).

A notion of weight of tableaux based on the tally statistic may provide an approach to the following concrete version of Open Problem~\ref{op:probl1}.
\begin{openproblem}Using the tally statistic on SYTs of shape $k\times n$, define a $ k$-dimensional generalization
of the weighted Catalan numbers that is equivalent to neither the $k$-dimensional SSWCNs nor the $k$-dimensional u-bounded SSWCNs.\end{openproblem}

\section*{Acknowledgments}

We thank Kenta Suzuki for suggesting the semisymmetric height function for the $k$-dimensional balanced ballot paths. We also thank Alexander Postnikov, Miroslav Marinov, and Tanya Khovanova for insightful conversations. The MIT Department of Mathematics financially supports the first author. The High School Student Institute of Mathematics and Informatics in Bulgaria supports the second author.
\bibliographystyle{plain}

\bibliography{References.bib}

\smallskip

\noindent
Ryota Inagaki \\
\textsc{
Department of Mathematics, Massachusetts Institute of Technology\\
77 Massachusetts Avenue, Building 2, Cambridge, Massachusetts, U.S.A. 02139}\\
\textit{E-mail address: }\texttt{inaga270@mit.edu}
\medskip

\noindent
Dimana Pramatarova \\
\textsc{
``Akademik Kiril Popov” High School of Mathematics }\\
\textit{E-mail address: }\texttt{dimanapramatarova@gmail.com}
\medskip
\end{document}